\documentclass[11pt]{amsart}
\usepackage{geometry}                
\usepackage{mathpazo,cite}  
\usepackage[usenames,dvipsnames,cmyk]{xcolor}
\usepackage{url}
\usepackage{amsmath,amssymb,latexsym,mathrsfs,amssymb,amsthm,amsfonts}
\usepackage{enumerate}
\usepackage{enumitem,overpic}
\usepackage{mathtools,mathdots} 
\usepackage[all]{xy}
\usepackage{tikz, tkz-euclide}
\usepackage{setspace}
\usepackage{bbm}
\usepackage{caption,subcaption}
\newcommand{\one}{\mathbbm{1}}
\newcommand{\bDelta}{\boldsymbol{\Delta}}

\renewcommand{\vec}{\mathbf}
\newcommand{\lvertiii}{\left\vert\kern-0.25ex\left\vert\kern-0.25ex\left\vert}
\newcommand{\rvertiii}{\right\vert\kern-0.25ex\right\vert\kern-0.25ex\right\vert}

\usepackage[colorlinks=true, linkcolor=MidnightBlue, citecolor=MidnightBlue, urlcolor=MidnightBlue]{hyperref}

\newtheorem{theorem}{Theorem}[section]
\newtheorem{lemma}[theorem]{Lemma}
\newtheorem{proposition}[theorem]{Proposition}
\newtheorem{corollary}[theorem]{Corollary}

\newtheorem{conjecture}{Conjecture}

\theoremstyle{definition}
\newtheorem{definition}[theorem]{Definition}

\theoremstyle{remark}
\newtheorem{remark}[theorem]{Remark}

\let\Im=\undefined\DeclareMathOperator{\Im}{Im}

\DeclareMathOperator{\id}{Id}
\newcommand{\D}{\mathrm{d}}

\title{The ultraspherical rectangular collocation method and its convergence}

\author{Thomas Trogdon}
\address{Department of Applied Mathematics, University of Washington, WA}
\email{trogdon@uw.edu}
\urladdr{http://faculty.washington.edu/trogdon/}

\begin{document}

\begin{abstract}
   We develop the ultraspherical rectangular collocation (URC) method, a collocation implementation of the sparse ultraspherical method of Olver \& Townsend for two-point boundary-value problems.  The URC method is provably convergent, the implementation is simple and efficient, the convergence proof motivates a preconditioner for iterative methods, and the modification of collocation nodes is straightforward.  The convergence theorem applies to all boundary-value problems when the coefficient functions are sufficiently smooth and when the roots of certain ultraspherical polynomials are used as collocation nodes.  We also adapt a theorem of Krasnolsel'skii et al.~to our setting to prove convergence for the rectangular collocation method of Driscoll \& Hale for a restricted class of boundary conditions.
\end{abstract}

\maketitle

\section{Introduction}
We consider the numerical solution of boundary-value problems on $\mathbb I := [-1,1]$ of the form
\begin{align}\label{eq:bvp}
\begin{split}
    \sum_{j=0}^{k} a_j(x) \frac{\D^j u}{\D x^j} (x) &= f(x), \quad x \in (-1,1),\\
    \vec S \begin{bmatrix}
        u(-1) \\
        u'(-1) \\
        \vdots \\
        u^{(k-1)}(-1)
    \end{bmatrix} + \vec T\begin{bmatrix}
        u(1) \\
        u'(1) \\
        \vdots \\
        u^{(k-1)}(1)
    \end{bmatrix} &= \vec b, \quad \vec S, \vec T \in \mathbb C^{k \times k}, \vec b \in \mathbb C^k.
    \end{split}
\end{align}
We develop an ultraspherical rectangular collocation (URC) method based on the sparse ultraspherical approach of Olver \& Townsend \cite{Olver2013} where the Galerkin projection on the range is simply replaced with collocation.  The approach incorporates the rectangular collocation ideas of Driscoll \& Hale \cite{Driscoll2015} (see also \cite{Aurentz2017}).  The method developed here has the following important features:
\begin{itemize}
    \item \textbf{The method is provably convergent.}  As far as we are aware, no collocation method for discretizing \eqref{eq:bvp} had been shown to converge in general.  In the current work we show that if $a_k = 1$: (1) With and a finite amount of smoothness of the coefficient functions and $f$, when using the roots of  ultraspherical polynomials as collocation nodes the method converges (see Theorem~\ref{t:informal}).  (2) If one uses the Chebyshev first-kind extrema or first-kind zeros as collocation nodes, the boundary conditions satisfy a regularity condition, and the coefficient functions  and $f$ are H\"older continuous, then the method converges at an optimal rate (see Theorem~\ref{t:restate}, a small adapation of \cite[Theorem 15.5]{Krasnoselskii1972}).
    \item \textbf{The implementation is efficient and simple.}  To efficiently implement the rectangular collocation method of Driscoll \& Hale \cite{Driscoll2015} and obtain an $O(N^2)$ complexity to construct an $N \times N$ linear system, one has to take care to iteratively construct differentiation matrices \cite{Xu2016}.  The URC method effectively requires only the use of a three-term recurrence for (normalized) ultraspherical polynomials to construct the differentiation matrices.
    \item \textbf{The method has an obvious preconditioner.} The proof of convergence of the URC method involves a two-sided preconditioning step.  We then show that the preconditioned matrix is close to the right-preconditioned finite-section matrix of Olver \& Townsend, which is, in the limit, of the form $\id + \vec K$ where $\vec K$ is compact.  The right preconditioner is diagonal and the left preconditioner is determined by the eigenvectors of the Jacobi matrix associated to (normalized) ultraspherical polynomials and is therefore reasonably efficient to implement.  For well-conditioned boundary-value problems, after preconditioning, we find an empirical $O(N^2)$ complexity to solve an $N \times N$ discretization of \eqref{eq:bvp} using GMRES \cite{GMRES-original}.
    \item \textbf{The discretization acts from coefficient space to value space.}  Historically, spectral collocation methods work by discretizing differentiation operators as mapping function values to function values \cite{TrefethenSpectral,Fornberg1996,Weideman2000,Boyd,TrefethenATAP}.  Here we advocate for a different approach when the solution of linear system associated to the discretization of \eqref{eq:bvp} results in the approximate orthogonal polynomial expansion coefficients of the unknown --- something we view as more useful output than function values. Indeed, for example, when one inputs a function into {\tt Chebfun} \cite{chebfun}, the first task is to compute its Chebyshev coefficients.
    \item \textbf{The choice of collocation nodes is simple to modify.}  The proofs of convergence for the URC method requires the use of zeros of ultraspherical polynomials as collocation nodes (Theorem~\ref{t:informal}) or the first-kind Chebyshev zeros or extrema (Theorem~\ref{t:restate}).  But, the user is free to choose any other choice of nodes with a simple modification of the method.  In our numerical experiments, we find that the using the roots of any ultraspherical polynomial produces comparable results to the zeros of the first-kind Chebyshev polynomials.  And the use of the extrema of Chebyshev first-kind polynomials produces slightly degraded results.
\end{itemize}

It is important to note that the method presented here does not match the complexity of Olver \& Townsend \cite{Olver2013} which achieves and $O(mN)$ complexity to solve \eqref{eq:bvp} when the coefficient functions are themselves polynomials of degree less than or equal to $m$.  The advantages of the collocation approach are largely implementational.  The collocation approach avoids the extra step of determining the expansion coefficients of the coefficient functions.  And the most simplistic implementation, avoids the basis conversion (connection coefficient) matrices. Coefficient functions with a finite amount of smoothness (i.e., derivatives at some order do not exist) are easier to handle with collocation, see Figure~\ref{fig:abs}.

\subsection{Outline of paper, main results and relation to previous work}

  Section~\ref{sec:OPs} is concerned with the absolute basics of the theory of orthogonal polynomials, Gaussian quadrature and its relation to interpolation, and the definition Jacobi polynomials.  Then Section~\ref{sec:Ultra} is concerned with theory specific to the ultraspherical (Gegenbauer) polynomials.   We find it convenient to work with orthonormal polynomials with respect to a normalized weight function (so that the zeroth-order polynomial is 1).  We first develop the differentiation operator, mapping between orthogonal polynomial bases and then define the polynomial evaluation matrices in Sections~\ref{sec:diff} and \ref{sec:eval}, respectively.  Importantly, this is all that is required to complete the derivation of the URC method.

Then in Sections~\ref{sec:connect} and \ref{sec:mult} we develop the matrix representations of basis conversion (connection coefficients) and function multiplication, respectively.  This then allows us to rederive the sparse ultraspherical method of Olver \& Townsend in Section~\ref{sec:usm}.  This full derivation is needed in our proof of convergence of the collocation method.  Then we close this section with useful estimates and properties of ultraspherical polynomials, see Section~\ref{sec:est}.

Our main theoretical developments are in Section~\ref{sec:conv} with many of the proofs deferred to Appendix~\ref{sec:defer}.  The main technical advance in this paper is presented in Section~\ref{sec:comp} which compares the left- and right-preconditioned collocation method with the right-preconditioned finite-section method using ideas from \cite{Trogdon2023}.  Lemma~\ref{l:aliasing} is used to estimate the effect of collocation and, as a result, our estimates only initially apply when the coefficient functions are degree $m$ polynomials and $m = o(N)$. Then Section~\ref{sec:conv-usm} essentially reviews the convergence proof of Olver \& Townsend, including estimates for truncations of polynomial expansions of the coefficient functions.  Section~\ref{sec:stab-coll} includes bounds for perturbations of the coefficient functions for the collocation method.  This allows us to remove the restriction of $m = o(N)$ and gives the main result of this paper in Section~\ref{sec:main}.  Loosely, speaking it states:

\begin{theorem}[Informal]\label{t:informal}
Let $t, \lambda > 0$, assume $a_k = 1$ and suppose that the roots of the $(k +\lambda)$th ultraspherical polynomials are used as collocation nodes. Then there exists $s,q> 0$ such that if $a_j \in C^q(\mathbb I)$, $j = 0,1,\ldots,k-1$, $f \in C^q(\mathbb I)$ and \eqref{eq:bvp} is uniquely solvable, then the difference of the solution of the collocation system and the true solution is $O(N^{-t})$ in $\ell^2_{s+k}$.
\end{theorem}

Then in Section~\ref{sec:demo} we present some numerical experiments. The code  to produce all the plots in this paper can be found here \cite{Trogdon2024}.  Section~\ref{sec:nodes} demonstrates the main theorem and explores the choice of collocation nodes.  Then Section~\ref{sec:precond} demonstrates that the proof of Theorem~\ref{t:main} is useful in educating preconditioners.  We finish the main text with some open questions in Section~\ref{sec:open}.

Appendix~\ref{sec:Kras} contains a modification of \cite[Theorem 15.5]{Krasnoselskii1972}, see Theorem~\ref{t:Kras}, which essentially states the following.

\begin{theorem}[Informal]\label{t:restate}
Assume $a_k = 1$ and suppose that the extrema or roots of the Chebyshev first-kind polynomials are used as collocation points\footnote{Other options are allowed, see Remark~\ref{r:zeros}.}. If $a_j \in C^{0,\alpha}(\mathbb I)$, $j = 0,1,\ldots,k-1$, $f \in C^{0,\alpha}(\mathbb I)$ for some $\alpha > 0$, \eqref{eq:bvp} is uniquely solvable and \eqref{eq:bvp} is uniquely solvable\footnote{This is an assumption on the boundary conditions.} if $a_j \equiv 0$ for $j < k$, then the difference of the solution of the collocation system and the true solution is bounded by the difference of the true solution and its interpolant.
\end{theorem}

These two theorems cover many cases for \eqref{eq:bvp} but both have their advantages and shortcomings.  First, Theorem~\ref{t:restate} allows for merely H\"older continuous coefficient functions and gives an optimal rate of convergence but does not allow for all possible boundary conditions.  For example, a second-order problem with Neumann boundary conditions does not fit into the framework.  Conversely, Theorem~\ref{t:informal} applies to any choice of boundary conditions. But the convergence theorem only applies if the coefficient functions are smooth enough.  And, the rate of convergence is not optimal.  It is an interesting, and apparently open, question to find a proof that closes these theoretical gaps.   We emphasize that they are just that, theoretical gaps, as the method, in both cases, performs very well.

\subsubsection{Relation to previous results}

As alluded to, the result of Krasnosel'skii et al.~\cite[Theorem~15.5]{Krasnoselskii1972} is the most general result we are aware of currently in the literature.  Yet, as discussed above, it does not, for example, establish convergence for second-order problems with Neumann boundary conditions, whereas the theorem in this work does not suffer from this gap.  It is possible that this shortcoming has been resolved in the references of \cite{Krasnoselskii1972}, but those references are either currently unavailable or not translated.

Other convergence results, of course, do exist in the literature for collocation methods.  For example, \cite{Zhang2008} considers second-order differential equations and establishes super-geometric rates of convergence.  The text \cite{Shen2011} establishes convergence for collocation methods for second-order differential equations.  Higher-order equations are treated in \cite{Shen2011} but theoretical results for these higher-order problems appear to be focused on Petrov--Galerkin methods. The text \cite{Canuto1988} also produces convergence results, generalizing the approach to multiple spatial dimensions, but focusing largely on elliptic operators. 

When it comes to preconditioning for spectral methods, there is a large literature.  We point to \cite{Lorentz1971, Du2015,Ross2023, Wang2014} for a discussion of so-called Birkhoff interpolation methods.  Such a method would be applicable when Theorem~\ref{t:restate} applies so that the leading-order operator can be inverted and used a preconditioner.  Essentially, the technique mirrors the proof of Theorem~\ref{t:Kras}.  While such an approach may achieve better results, once the implementational complexity gets to this level, its likely one should resort to the full implementation of Olver \& Townsend with the adaptive QR procedure. The preconditioning we propose here is straightforward and universal --- it applies whenever Theorem~\ref{t:informal} applies.

From a purely implementational perspective, we are unaware of other works using the approach of Olver \& Townsend to simplify and unify implementations.

\subsection{Notation}

We end this section with the introduction of some necessary notation.  The sequence spaces $\ell^p_s(\mathbb N)$ are defined by
\begin{align*}
    \ell_s^p(\mathbb N) := \left\{ \vec v = (v_j)_{j=1}^\infty ~|~ \|\vec v\|_{\ell^p_s}^p := \sum_{j=1}^\infty j^{ps} |v_j|^p < \infty \right\}.
\end{align*}
Such spaces for vectors $\vec v = (v_j)_{j=1}^N$, $N < \infty$ with the obvious norm will also be considered.  The domain of $\ell_s^p$ is omitted when it is clear from context.  If $s = 0$, $\ell^p$ is used to refer to these spaces.  Then $\| \diamond \|_{\rm F}$ is used to denote the Frobenius (Hilbert--Schmidt) norm\footnote{The notation $\diamond$ is used to refer to independent arguments of a function, i.e., $f(x) = 1/x$ can be denoted by $1/\diamond$.} and $\|\diamond\|_\infty$ denotes the standard max norm on $\mathbb I$.  And $C^q(\mathbb I)$ will is used to denote the space of complex-valued functions with $q$ continuous derivatives, $C^0(\mathbb I) = C(\mathbb I)$, with norm
\begin{align*}
\|f\|_{C^{q}(\mathbb I)}  = \sum_{\ell = 0}^q \|f^{(\ell)}\|_\infty.
\end{align*}
The notation $C^{q,\alpha}(\mathbb I)$ denotes the space of functions with $q$ continuous derivatives such that the $q$th derivative is $\alpha$-H\"older with exponent $0 < \alpha \leq 1$ with norm
\begin{align*}
    \|f\|_{C^{q,\alpha}(\mathbb I)} = \|f\|_{C^{q}(\mathbb I)} + \sup_{x \neq y} \frac{|f^{(q)}(x) - f^{(q)}(y)|}{|x - y|^{\alpha}}.
\end{align*}
Then $\id$ will be used to denote the identity operator/matrix depending on the context and $\id_n$ is the $n\times n$ identity matrix.

To state the next definition, $H^k(\mathbb I)$ denotes the space of measurable functions such that $f,f',\ldots,f^{(k-1)}$ are absolutely continuous and $f,f',\ldots,f^{(k)}$ are square integrable with the norm
\begin{align*}
    \|f\|_{H^k(\mathbb I)}^2 = \sum_{\ell=0}^k \|f^{(\ell)}\|_{L^2(\mathbb I)}^2.
\end{align*}

\begin{definition}
    We say \eqref{eq:bvp} is uniquely solvable if the only function in $H^k(\mathbb I)$ that solves the boundary-value problem with $\vec b = \vec 0$, $f = 0$, is the trivial solution.
\end{definition}

\section{Orthogonal polynomials}\label{sec:OPs}

In the current work, we work with orthonormal polynomials on $\mathbb I = [-1,1]$ but include some general developments.  Interested readers in the general theory of orthogonal polynomials are referred to \cite{GautschiOP} and \cite{Szego1939}.  For a (Borel) probability measure $\mu$ on $\mathbb R$, define the inner product
\begin{align*}
    \langle f, g \rangle_\mu := \int_{\mathbb R} f(x) \overline{g(x)} \mu(\D x).
\end{align*}
Then $L^2(\mu)$ is used to denote the Hilbert space with this inner product.  Since polynomials are often dense in $L^2(\mu)$ one can perform the Gram-Schmidt process on the monomials $\{1, \diamond, \diamond^2, \ldots \}$ using the inner product to obtain an orthonormal basis for $L^2(\mu)$.  Often, this process is described by first constructing the monic orthogonal basis $(\pi_k(\diamond;\mu))_{k \geq 0}$ satisfying
\begin{itemize}
    \item $\pi_k(x;\mu) = x^k + O(x^{k-1})$, \quad $x \to \infty$, and
    \item $\int \pi_k(x;\mu) \pi_j(k;\mu) \mu(\D x) = 0$ for $j \neq k$.
\end{itemize}
The orthonormal polynomials are defined by
\begin{align*}
    p_k(x;\mu) = \frac{\pi_k(x;\mu)}{\|\pi_k\|_{L^2(\mu)}}, \quad k = 0,1,2,\ldots.
\end{align*}
Arguably the most fundamental aspect of orthogonal polynomials is the symmetric three-term recurrence that they satisfy:
 \begin{align*}
        x p_j(x;\mu) = a_j(\mu) p_j(x;\mu) + b_{j-1}(\mu) p_{j-1}(x;\mu) + b_j(\mu) p_{j+1}(x;\mu), \quad j \geq 0,
\end{align*}
for sequence $a_j(\mu),b_j(\mu)$, $j \geq 0$ where $b_j(\mu) > 0$ for $j \geq 0$ and $b_{-1}(\mu) = p_{-1}(x;\mu) = 0.$

\begin{definition}
Let $\mu$ be a Borel measure on $\mathbb R$ with an infinite number of points in its support\footnote{This is necessary to ensure that the orthogonal polynomial sequence is infinite.}.  Then define the Jacobi operator
\begin{align*}
    \vec J(\mu) = \begin{bmatrix} a_0(\mu) & b_0(\mu) \\
    b_0(\mu) & a_1(\mu) & b_1(\mu)\\
    & b_1(\mu) & a_2(\mu) & \ddots\\
    && \ddots & \ddots \end{bmatrix}.
\end{align*}
Finite truncations are referred to as Jacobi matrices:
     \begin{align*}
        \vec J_N(\mu) := \begin{bmatrix} a_0(\mu) & b_0(\mu) \\
    b_0(\mu) & a_1(\mu) & b_1(\mu)\\
    & b_1(\mu) & a_2(\mu) & \ddots\\
    && \ddots & \ddots &  b_{N-2}(\mu) \\
    & & & b_{N-2}(\mu) & a_{N-1}(\mu)\end{bmatrix}.
    \end{align*}
\end{definition}

\subsection{Gaussian quadrature}

We now include a brief discussion of the development of Gaussian quadrature rules.  A quadrature rule on $\mathbb R$ consists of a set of nodes $x_1 < x_2 < \cdots < x_N$ and weights $w_j$, $j = 1,2,\ldots,N$ such that, informally,
\begin{align*}
    \int_{\mathbb R} f(x) \mu(\D x) \approx \sum_j w_j f(x_j).
\end{align*}
The latter expression can be identified with a measure $\sum_j w_j \delta_{x_j}$. We write
\begin{align*}
    E_N(f) = E_N(f;(x_j),(w_j)) = \int_{\mathbb R} f(x) \mu(\D x) - \sum_j w_j f(x_j).
\end{align*}
A quadrature formula is said to have degree of exactness $d$ if
\begin{align*}
    E_N(p) = 0, \quad \forall p \in \rm{span}\{1, \diamond, \ldots,\diamond^d\}.
\end{align*}

While there are many ways to motivate the following, for us, the definition of a Gaussian quadrature rule for a measure $\mu$ comes from the following observation from inverse spectral theory.  For convenience, suppose that $\mu$ has compact support, then \cite{DeiftOrthogonalPolynomials}
\begin{align*}
    \int_{\mathbb R} \frac{\mu(\D x)}{x - z} = \vec e_1^T (\vec J(\mu) - z)^{-1} \vec e_1, \quad \Im z > 0.
\end{align*}
If we instead considered the finite truncation $\vec J_N$, using its eigenvalue decomposition,
\begin{align*}
\vec U = \begin{bmatrix} \vec u_1 & \vec u_2 & \cdots & \vec u_N \end{bmatrix},\quad 
    \vec \Lambda = \mathrm{diag}(\lambda_1, \ldots, \lambda_N),
\end{align*}
we find
\begin{align*}
    \vec e_1^T (\vec J_N - z)^{-1} \vec e_1 = \sum_{j=1}^N \frac{w_j}{\lambda_j -z}.
\end{align*}
We recognize the latter as
\begin{align*}
    \sum_{j=1}^N \frac{w_j}{\lambda_j -z} = \int_{\mathbb R} \frac{\mu_N(\D x)}{x - z}, \quad \mu_N = \sum_{j=1}^N |u_{1j}|^2 \delta_{\lambda_j},
\end{align*}
where $u_{ij}$ is the $(i,j)$ entry of $\vec U$. We call this $\mu_N$ the $N$th-order Gaussian quadrature rule for $\mu$ and it is well-known that it has degree of exactness $2N-1$ \cite{GautschiOP}.

\subsection{Interpolation}

Given a measure $\mu$ on $\mathbb I$ and its Jacobi operator $\vec J(\mu)$, the Gaussian quadrature rules associated to it provide a natural way to discretize the inner product $\langle \diamond,\diamond \rangle_\mu$:
\begin{align*}
    \langle f, g \rangle_{\mu,N} = \int_{\mathbb R} f(x) \overline{g(x)} \mu_N (\D x) = \sum_{j=1}^N f(\lambda_j) \overline{g(\lambda_j)} w_j.
\end{align*}
Define
\begin{align*}
    {\mathcal I}^\mu_N f (x) = \sum_{j=0}^{N-1} \langle f, p_j(\diamond;\mu) \rangle_{\mu,N}\, p_j(x;\mu).
\end{align*}
 We include the (classical) proof of the following because it requires the definition of quantities that will be of use in what follows.  See \cite{GautschiOP} for  more detail.
\begin{theorem}
   Consider a probability measure $\mu$ with $\mathrm{supp}(\mu) = \mathbb I$ and its Jacobi operator $\vec J(\mu)$, let $\lambda_j$, $j=1,2,\ldots,N$ denote the eigenvalues of $\vec J_N(\mu)$.  Then
    \begin{align*}
        {\mathcal I}^\mu_N f (\lambda_j) = f(\lambda_j), \quad j = 1,2,\ldots,N.
    \end{align*}
\end{theorem}
\begin{proof}
It is well-known that the orthonormal matrix of eigenvectors is given by
    \begin{align}\label{eq:Umat}
        \vec U_N(\mu) = \underbrace{\begin{bmatrix}
        \ddots & \vdots & \iddots\\
        \cdots & p_j(\lambda_\ell;\mu) & \cdots \\
        \iddots & \vdots & \ddots \end{bmatrix}}_{\vec P_N(\mu)} \vec W_N(\mu),
    \end{align}
    where $\vec W_N(\mu)$ is chosen to normalize the columns, and the index $j$ refers to rows while $\ell$ refers to columns.   Then
    \begin{align}\label{eq:weights}
        w_j =  w_j(\lambda,N) = \left( \sum_{\ell = 0}^{N-1} p_\ell(\lambda_j;\mu)^2 \right)^{-1}, \quad \vec W_N(\mu) = {\rm diag}(\sqrt w_1, \sqrt w_2,\ldots,\sqrt w_N).
    \end{align}
    Upon setting $c_j = \langle f, p_j(\diamond;\mu) \rangle_{\mu,N}$, we find
    \begin{align}\label{eq:FN}
        \begin{bmatrix} c_0 \\ \vdots \\ c_{N-1} \end{bmatrix} = \underbrace{\vec P_N(\mu) \vec W_N(\mu)^2}_{\vec F_N(\mu)} \begin{bmatrix} f(\lambda_1) \\ \vdots \\ f(\lambda_N) \end{bmatrix} = \vec U_N(\mu) \vec W_N(\mu) \begin{bmatrix} f(\lambda_1) \\ \vdots \\ f(\lambda_N) \end{bmatrix}.
    \end{align}
    Then because $\vec U_N(\mu)$ must be orthogonal, we find 
    \begin{align*}
        \begin{bmatrix} f(\lambda_1) \\ \vdots \\ f(\lambda_N) \end{bmatrix}  =   \vec W_N(\mu)^{-1} \vec U_N(\mu)^T \begin{bmatrix} c_0 \\ \vdots \\ c_{N-1} \end{bmatrix} =   \vec P_N(\mu)^T\begin{bmatrix} c_0 \\ \vdots \\ c_{N-1} \end{bmatrix}.
    \end{align*}
\end{proof}

\subsection{Jacobi polynomials}

In this section we highlight properties of the orthogonal polynomials with respect to the two-parameter family of weight functions
\begin{align}\label{eq:jac-meas}
  w_{\alpha,\beta}(x) :=  Z^{-1}_{\alpha,\beta} (1 - x)^\alpha (1 + x)^\beta \one_{[-1,1]}(x), \quad \alpha,\beta > -1.
\end{align}
Here $Z_{\alpha,\beta}$ is the normalization constant so that
\begin{align*}
    \mu(\D x) = w_{\alpha,\beta}(x) \D x,
\end{align*}
is a probability measure on $\mathbb R$ and can be computed as
\begin{align*}
    Z_{\alpha,\beta} = \frac{_2F_1(1, -\alpha, 2 + \beta, -1)}{1 + \beta} + \frac{_2F_1(1, -\beta, 2 + \alpha, -1)}{1 + \alpha},
\end{align*}
in terms of the hypergeometric function $_2F_1$ \cite{DLMF}.  We use $p_j(x;\alpha,\beta)$ to refer to the $j$th orthonormal polynomial and $a_j(\alpha,\beta), b_j(\alpha,\beta)$ to refer to the recurrence coefficients.  The polynomial $p_j(x;\alpha,\beta)$ is called an orthonormal Jacobi polynomial.  The classical notation \cite{DLMF} is for unnormalized, and not monic, Jacobi polynomials is $P_j^{(\alpha,\beta)}(x)$ such that
\begin{align*}
    m_j(\alpha,\beta) := \int_{-1}^1 P_j^{(\alpha,\beta)}(x)^2 (1 - x)^\alpha (1 + x)^\beta \D x = \frac{ 2^{\alpha + \beta + 1} \Gamma( j + \alpha + 1) \Gamma(j + \beta + 1)}{(2 j + \alpha + \beta +1) \Gamma(j + \alpha + \beta + 1) j!},
\end{align*}
where $\Gamma(\diamond)$ is the Gamma function \cite{DLMF}.  Set $d_j = d_j(\alpha,\beta) = 2j + \alpha+ \beta$.  The polynomials satisfy the three-term recurrence relation
\begin{align*} 2j& (j + \alpha + \beta)(d_j -2) P_j^{(\alpha,\beta)}(x) \\
&= (d_j -1) \left[ d_j(d_j-2)x + \alpha^2 - \beta^2 \right] P_{j-1}^{(\alpha,\beta)}(x) - 2d_j (d_j - \beta - 1) (d_j - \alpha - 1) P_{j-2}^{(\alpha,\beta)}(x).
\end{align*}

\subsubsection{The Jacobi operator}
It follows that, for $j = 1,2,\ldots$
\begin{align*} b_{j-1}(\alpha,\beta) = \frac{2\sqrt{j}\sqrt{(j + \alpha)(j + \beta)} \sqrt{j + \alpha + \beta}}{d_j \sqrt{d_j^2 - 1}  }, \quad  a_{j-1}(\alpha,\beta) = \frac{\beta^2 - \alpha^2}{d_j(d_j-2)}.
\end{align*}
If any of these expressions are $0/0$ indeterminate, the issue can be resolved by fixing $j$ and taking a limit as $\alpha,\beta$ approach the desired value.

\section{Ultraspherical (Gegenbauer) methods}\label{sec:Ultra}

The classical ultraspherical polynomials, denoted by $C_j^{(\lambda)}(x)$, which are orthogonal with respect to $\mu_\lambda (\D x) \propto w_\lambda(x) \D x$,  $w_\lambda(x) = (1 - x^2)^{\lambda - \frac 1 2}$, are not orthonormal \cite{DLMF}. For convenience, define
\begin{align*}
p_j(x;\lambda) = p_j(x;\lambda - 1/2,\lambda-1/2), \quad \pi_j(x;\lambda) = \pi_j(x;\lambda - 1/2,\lambda-1/2).
\end{align*}
And we use the notation $a_j(\lambda) = a_j(\lambda - 1/2, \lambda - 1/2)$, $b_j(\lambda) = b_j(\lambda - 1/2, \lambda - 1/2)$.

Consider some quantities
\begin{align*}
    Z_\lambda &:= \int_{-1}^1 w_\lambda(x) \D x = \sqrt{\pi} \frac{\Gamma(\lambda + \frac 1 2)}{\Gamma(\lambda + 1)}, \quad 
    \tilde w_\lambda(x) = Z_\lambda^{-1} w_\lambda(x), \quad \mu_\lambda(\D x) = \tilde w_\lambda(x) \D x,\\
     k_j(\lambda) &:= \frac{2^j (\lambda)_j}{j!}, \quad 
     h_j(\lambda) := \frac{2^{1 - 2 \lambda} \pi \Gamma(j + 2 \lambda)}{(j + \lambda) \Gamma(\lambda)^2 j!}.
\end{align*}
Then define
\begin{align*}
    c_j(\lambda) =  \frac{\Gamma(\lambda + 1)}{ \sqrt{\pi}\Gamma(\lambda + \frac 1 2)} \frac{h_j(\lambda)}{k_j(\lambda)^2}, \quad \text{and} \quad
    p_j(x;\lambda) = \frac{1}{\sqrt{c_j(\lambda)}} \pi_j(x;\lambda),
\end{align*}
is appropriately normalized.

\subsection{Differentiation}\label{sec:diff}
It follows directly that the monic ultraspherical polynomials satisfy
\begin{align*}
    \pi_j'(x;\lambda) = j \pi_{j-1}(x;\lambda + 1), \quad 
    p_j'(x;\lambda) = j \sqrt{\frac{c_{j-1}(\lambda + 1)}{c_{j}(\lambda)}} p_{j-1}(x;\lambda + 1).
\end{align*}
The leads us to define
\begin{align*}
    \vec D_{\lambda \to \lambda +1} = \begin{bmatrix} 
    0 &  d_1(\lambda)\\
     & 0  &d_2(\lambda) \\
    & & 0 & d_3(\lambda) \\
    & & & 0  & \ddots \\
    & & & & \ddots
    \end{bmatrix}, \quad d_j(\lambda) := j \sqrt{\frac{c_{j-1}(\lambda + 1)}{c_{j}(\lambda)}} = j \sqrt{ \frac{2 (\lambda +1) (j+2 \lambda )}{2 j \lambda +j}}
\end{align*}
and
\begin{align*}
    \vec D_k(\lambda) = \vec D_{\lambda+k-1 \to \lambda+k} \cdots \vec D_{\lambda+1 \to \lambda+2} \vec D_{\lambda \to \lambda+ 1}, \quad \vec D_0 = \id.
\end{align*}
Thus, if $\vec c = (c_j)_{j \geq 0}$ are such that, formally,
\begin{align*}
    u(x) = \sum_j c_j p_j(x;\lambda),
\end{align*}
then for $\vec d = \vec D_k(\lambda) \vec c = (d_j)_{j \geq 0}$
\begin{align}
    u^{(k)}(x) = \sum_j d_j p_j(x;\lambda+k).
\end{align}

\subsection{Evaluation}\label{sec:eval}

The three-term recurrence can be used to evaluate an orthogonal polynomial series.  When the (finite number of) coefficients are known, Clenshaw's algorithm \cite{Clenshaw1955} is typically thought of as the best way to evaluate the series (see also \cite{Olver2020}), but if the coefficients are unknown --- they are the solution of a linear system --- we use the recurrence.  

Specifically, let $P = (x_1,\ldots,x_m)$ be a grid on $\mathbb I$.  Then define the evaluation matrix
\begin{align*}
    \vec P_{\lambda \to P} = \begin{bmatrix}
        p_0(x_1;\lambda) & p_1(x_1;\lambda) & p_2(x_1;\lambda) & \cdots \\
        p_0(x_2;\lambda) & p_1(x_2;\lambda) & p_2(x_2;\lambda) & \cdots \\
        \vdots & \vdots & \vdots\\
        p_0(x_m;\lambda) & p_1(x_m;\lambda) & p_2(x_m;\lambda) & \cdots \\
    \end{bmatrix}.
\end{align*}
Depending on the context, we might take $\vec P_{\lambda \to P}$ to have either a finite or an infinite number of columns. Then
\begin{align*}
    \begin{bmatrix} u^{(k)}(x_1) \\ \vdots \\  u^{(k)}(x_m) \end{bmatrix} = \vec P_{k + \lambda \to P} \vec D_k(\lambda) \vec c.
\end{align*}
To construct $\vec P_{\lambda \to P}$,  observe that the columns $\vec p_j$ satisfy the three-term recurrence:
\begin{align*}
    \vec p_{j+1} = \frac{1}{b_j(\lambda)} \left[ \vec x \cdot \vec p_j - a_j(\lambda) \vec p_j - b_{j-1}(\lambda) \vec p_j\right], \quad \vec p_{-1} = \vec 0, \quad \vec p_{0} = \vec 1.
\end{align*}
where $\vec x = (x_j)_{j=1}^m$ and $\cdot$ denotes the entrywise product.  This gives us all the tools required to solve \eqref{eq:bvp} using collocation.

\subsection{The URC method}\label{sec:coll}

We now use collocation to solve \eqref{eq:bvp} motivated by \cite{Driscoll2015}. To impose boundary conditions, if $u(x) = \sum_j u_j p_j(x;\lambda)$, $\vec u = (u_j)$ then
\begin{align*}
    \begin{bmatrix}
        u(a) \\
        u'(a) \\
        \vdots \\
        u^{(k-1)}(a)
    \end{bmatrix} = \underbrace{\begin{bmatrix} \vec P_{\lambda  \to \{a\}} \vec D_0(\lambda) \\ \vec P_{\lambda +1 \to \{a\}} \vec D_1(\lambda) \\ \vdots \\ \vec P_{\lambda+k-1 \to \{a\}} \vec D_{k-1}(\lambda) \end{bmatrix}}_{\vec E_a(\lambda,k)} \vec u.
\end{align*}

Let $P = (x_1,\ldots,x_{N -k})$ be a grid on $(-1,1)$, and define
\begin{align*}
    a_j(P) = \mathrm{diag}(a_j(x_1),\ldots,a_j(x_{N-k})).
\end{align*}
Set
\begin{align*}
\vec L_P = \sum_{j = 0}^k a_j(P)\vec P_{\lambda + j \to P} \vec D_j(\lambda).
\end{align*}
The discretized $N \times N$ collocation system is simply given by
\begin{align}\label{eq:coll}
    \vec L_{N}^{\rm{C}} \tilde{\vec u}_N &=  \begin{bmatrix} \vec S \vec E_{-1}(\lambda,k) + \vec T \vec E_{1}(\lambda,k)  \\
    \vec L_P \end{bmatrix} \vec Q_N \tilde{\vec u}_N = \begin{bmatrix}
        \vec b\\
        f(x_1) \\
        \vdots \\
        f(x_{N-k})
    \end{bmatrix},\quad \vec Q_N = \begin{bmatrix} \id_N \\ \vec 0 \\ \vdots \end{bmatrix}, \quad 
    \vec L_{N}^{\rm{C}} = \vec L_{N}^{\rm{C}}(a_0,\ldots,a_k).
\end{align}

\subsection{Connection coefficients (basis conversion)}\label{sec:connect}

In the following, we will need to convert an expansion in $p_j(x;\lambda)$ to one in $p_j(x;\lambda + 1)$ and we, of course, use connection coefficients for this purpose.  Write
\begin{align*}
    p_k(x;\lambda) = \sum_{j=0}^k c_{k,j} p_j(x;\lambda +1), \quad c_{k,j} = \int_{-1}^1 p_j(x;\lambda) p_k(x;\lambda +1) \tilde w_{\lambda+1}(x) \D x.
\end{align*}
It follows that this vanishes for $j < k$, by orthogonality of $p_k(x;\lambda + 1)$.  Furthermore, for $k > j + 2$, the orthogonality of $p_k(x;\lambda)$ and $(1 -x^2) p_j(x;\lambda + 1)$ implies this vanishes.  So, it remains to compute, for $k > 0$:
\begin{align*}
    \int_{-1}^1 p_k(x;\lambda+1) p_k(x;\lambda) \tilde w_{\lambda+1}(x) \D x  & =  \frac{\sqrt{c_k(\lambda+1)}}{\sqrt{c_k(\lambda)}},\\
    \int_{-1}^1 p_{k-1}(x;\lambda+1) p_k(x;\lambda) \tilde w_{\lambda+1}(x) \D x & = 0,\\
    \int_{-1}^1 p_{k-2}(x;\lambda+1) p_k(x;\lambda) \tilde w_{\lambda+1}(x) \D x & = -  \frac{Z_\lambda}{Z_{\lambda +1}} \frac{\sqrt{c_{k}(\lambda)}}{\sqrt{c_{k-2}(\lambda+1)}}.
\end{align*}
We then obtain the simplified relations
\begin{align*}
    \frac{\sqrt{c_k(\lambda+1)}}{\sqrt{c_k(\lambda)}} &= \sqrt{\frac{(\lambda +1) (k+2 \lambda ) (k+2 \lambda +1)}{2 (2 \lambda +1) (k+\lambda )(k+\lambda +1)}},\\
    \frac{Z_\lambda}{Z_{\lambda +1}} \frac{\sqrt{c_{k}(\lambda)}}{\sqrt{c_{k-2}(\lambda+1)}} &= \sqrt{\frac{(k-1) k (\lambda +1)}{2 (2 \lambda +1) (k+\lambda -1) (k+\lambda )}}.
\end{align*}
Define
\begin{align*}
    s_k(\lambda) &:= \begin{cases} 
        1 & k = 0,\\
         \sqrt{\frac{(\lambda +1) (k+2 \lambda ) (k+2 \lambda +1)}{2 (2 \lambda +1) (k+\lambda )(k+\lambda +1)}} & \text{otherwise}, \end{cases} \quad 
    t_k(\lambda) 
         := \sqrt{\frac{(k-1) k (\lambda +1)}{2 (2 \lambda +1) (k+\lambda -1) (k+\lambda )}}, 
\end{align*}
and
\begin{align*}
    \vec C_{\lambda \to \lambda +1} = \begin{bmatrix}
    s_0(\lambda) & 0 & - t_2(\lambda) &   \\
    & s_1(\lambda) &  0  & - t_3(\lambda) &  \\
    & & s_2(\lambda)  & 0 & - t_4(\lambda) \\
    & & & \ddots & \ddots & \ddots 
    \end{bmatrix}.
\end{align*}
Therefore if $\vec d = \vec C_{\lambda \to \lambda +1} \vec c$ then, formally,
\begin{align*}
    \sum_j d_j p_j(x;\lambda + 1) = \sum_j c_j p_j(x,\lambda).
\end{align*}
And we use the notation 
\begin{align*}
 \vec C_{\lambda \to \lambda + k} = \vec C_{\lambda + k -1 \to \lambda +k } \cdots \vec C_{\lambda \to \lambda +1 }, \quad \vec C_{\lambda \to \lambda} = \id.
\end{align*}

\subsection{Function multiplication}\label{sec:mult}

To handle multiplication as an operator on coefficients, we will suppose that our input coefficients have rapidly converging orthogonal polynomial expansions.  But first, assume a finite expansion
\begin{align*}
    q(x) = \sum_{j=0}^{m} \alpha_j p_j(x;0).
\end{align*}
An expansion in a different orthogonal polynomial basis can be assumed, and the derivation below generalizes straightforwardly by replacing the recurrence coefficients in \eqref{eq:matrixpoly} appropriately. Then $\vec J(\lambda)$ encodes multiplication by $x$:
\begin{align*}
    u(x) = \sum_{j} u_j p_j(x;\lambda), \quad \vec v = \vec J(\lambda) \vec u, \quad x u(x) = \sum_j v_j p_j(x;\lambda),
\end{align*}
and therefore
\begin{align*}
    q(x) u(x) = \sum_{j} w_j p_j(x;\lambda), \quad \vec w = q(\vec J(\lambda)) \vec u.
\end{align*}

We need to develop (stable) methods to evaluate $ q(\vec J(\lambda)) \vec u$ or $ q(\vec J(\lambda))$.  To evaluate the latter, we will be able to replace $\vec u$ with an identity matrix.  The following gives the recurrence
\begin{align}\begin{split}\label{eq:matrixpoly}
    \vec p_0 &= \vec u,\\
    \vec p_1 &= \sqrt{2} \vec J(\lambda) \vec p_0,\\
    \vec p_2 &= 2 \vec J(\lambda) \vec p_{1} - \sqrt{2} \vec p_0,\\
    \vec p_j &= 2 \vec J(\lambda) \vec p_{j-1} - \vec p_{j-2}, \quad j \geq 3,    
\end{split}\end{align}
which is run simultaneously with the iterates
\begin{align*}
    \vec q_{-1} &= \vec 0,\\
    \vec q_{j} &= \vec q_{j-1} + \alpha_j \vec p_j, \quad 0 \leq j \leq m,
\end{align*}
and $\vec w = \vec q_m$.  We denote by $\vec M(q;\lambda)$ the resulting operator ($\vec u = \id$) when $m$ is finite, the limit of $\vec q_m$, if it exists, if $m = \infty$.

\subsection{The sparse ultraspherical method}\label{sec:usm}
We are now in a place to describe the sparse ultraspherical spectral method of Olver \& Townsend.  The method works by constructing a semi-infinite matrix representation of \eqref{eq:bvp}.  Specifically, the Petrov--Galerkin projections give
\begin{align*}
    \mathcal L = \sum_{j=0}^k a_j(x) \frac{\D^k}{\D x^k} \to \vec L : =\sum_{j=0}^k \vec M(a_j;k + \lambda)\vec C_{j + \lambda \to k + \lambda} \vec D_j(\lambda).
\end{align*}
Here the domain of $\vec L$ should be thought of as the expansion coefficients for a function in a $p_j(x;\lambda)$ series. A common choice is $\lambda = 0$.  Some symmetry properties can be maintained if one choose $\lambda = 1/2$ \cite{Aurentz2020}.  

Then we suppose that $f(x) = \sum_j f_j p_j(x;\lambda)$, $\vec f = (f_j).$  The full system for the unknown $\vec u$ becomes
\begin{align}\label{eq:PGM}
    \begin{bmatrix} \vec S \vec E_{-1}(\lambda,k) + \vec T \vec E_{1}(\lambda,k) \\
    \vec L \end{bmatrix} \vec u = \begin{bmatrix}
        \vec b\\
        \vec C_{\lambda \to \lambda +k} \vec f
    \end{bmatrix} =: \vec y.
\end{align}
If the coefficient functions $a_j$ are low-degree polynomials this system is very sparse.  Many methods can be employed to solve it, including: (1) finite-section truncations, (2) an adaptive QR procedure \cite{Olver2013} and (3) iterative methods after preconditioning.

\subsection{Ultraspherical estimates}\label{sec:est}


In order to establish our convergence result, we will need some fairly detailed estimates on ultraspherical polynomials. The first result is a useful upper bound, see \cite{Forster1993}.
\begin{lemma}\label{l:bound}
    For $\lambda \geq 0$, there exists $c(\lambda)$ such that
    \begin{align*}
        |(\sin \theta)^\lambda p_j(\cos \theta;\lambda)| \leq c(\lambda), \quad j = 1,2,\ldots.
    \end{align*}
\end{lemma}
\begin{proof}
    From \cite{ErvandKogbetliantz1924}, see also \cite{Forster1993}, we have
    \begin{align*}
        |(\sin \theta)^\lambda C^{(\lambda)}_j(\cos \theta)| \leq 2 \frac{\Gamma(j + \lambda)}{\Gamma(\lambda) \Gamma(j +1)}, \quad j = 1,2,\ldots.
    \end{align*}
    where $C^{(\lambda)}_j(x) \propto p_j(x;\lambda)$ is the ultraspherical polynomials as given in \cite{DLMF}.  This does not give these polynomials the same normalization as $P_j^{(\lambda - 1/2, \lambda -1/2)}$.  Then
    \begin{align*}
        \int_{-1}^1 C_j^{(\lambda)}(x)^2 \tilde w_\lambda(x) \D x = \frac{2^{1 - 2 \lambda} \pi \Gamma(j + 2 \lambda)}{(j + \lambda) (\Gamma(\lambda))^2 j!}  \frac{\Gamma(\lambda + 1)}{\sqrt{\pi}\Gamma(\lambda + \frac 1 2)}.
    \end{align*}
    So, we find that
    \begin{align*}
        p_j(x;\lambda) = c_j^{(\lambda)} C_j^{(\lambda)}(x), \quad c_j^{(\lambda)} = \sqrt{\frac{(j + \lambda)j!}{\Gamma(j + 2 \lambda)}} h(\lambda).
    \end{align*}
    Then it follows from Stirling's approximation that as $j \to \infty$
    \begin{align*}
        \sqrt{\frac{(j + \lambda)j!}{\Gamma(j + 2 \lambda)}} \frac{\Gamma(j + \lambda)}{\Gamma(j +1)} = 1 + o(1),
    \end{align*}
    and the claim follows.
\end{proof}

The next result concerns the behavior of the matrix $\vec W_N(\mu)$ and can be found in \cite{Petras1996}.  
\begin{lemma}\label{l:weights}
    Suppose $\lambda > -1/2$, and let $x_{1}(\lambda,N) < x_{2}(\lambda,N) < \cdots < x_{N}(\lambda,N)$ be the roots of $p_N(x,\lambda)$.  Then
    \begin{align*}
        w_j(\lambda,N) = Z_\lambda^{-1} \frac{\pi}{N} (1 - x_j^2)^\lambda ( 1 + O(N^{-2} (1 - x_j^2)^{-1})).
    \end{align*}
\end{lemma}

To make full use of this result, we need asymptotics for the extreme roots of $p_N(x;\lambda)$.  By symmetry, it suffices to consider just one.  The following is from \cite{KuijlaarsInterval}:
\begin{align*}
x_{1}(\lambda,N) = -1 + c_\lambda N^{-2} + O(N^{-3}), \quad c_\lambda > 0.
\end{align*}
This establishes that the error term in Lemma~\ref{l:weights} is $O(1)$.  Therefore, we have the following:

\begin{lemma}[Aliasing estimate]\label{l:aliasing}
For $\lambda > -1/2$ there exists $C(\lambda) > 0$, independent of $N$, such that
\begin{align*}
    |\langle p_i(\diamond,\lambda), p_{j}(\diamond,\lambda) \rangle_{\mu_\lambda,N}| \leq C(\lambda),
\end{align*}
for all $i,j$.
\end{lemma}
And then we have another useful, yet crude, bound from \cite[Theorem 7.32.1]{Szego1939}, after accounting for normalizations.
\begin{lemma}\label{l:opgrowth}
For $\lambda \geq 0$, there exists $\ell(\lambda)$ such that
\begin{align*}
    \|p_j(\diamond;\lambda)\|_\infty \leq \ell(\lambda) (j+1)^\lambda, \quad j \geq 0.
\end{align*}
\end{lemma}

\section{Convergence}\label{sec:conv}

The proof of convergence for the collocation method \eqref{eq:coll} with ultraspherical polynomial roots as collocation nodes proceeds in three main steps.
\begin{enumerate}
    \item First, we compare \eqref{eq:coll} with finite sections of \eqref{eq:PGM} when the coefficient functions are polynomials of slowly growing degree.  To effectively compare the operators involved, we have to use both left and right `preconditioners'.  Here Lemma~\ref{l:aliasing} plays a crucial role.
    \item Then, we effectively review the convergence proof of Olver \& Townsend and introduce stability estimates to understand the effect of approximating coefficient functions with polynomials.
    \item Lastly, we use another stability estimate to understand the effect of replacing coefficient functions with polynomials in the collocation method.
\end{enumerate}

\subsection{Preliminaries}\label{sec:prelim}

We first need to study the regularity of the coefficient functions and its effect on the operators $\vec M(a_j,\lambda + k)$.  We consider weighted norms, so we introduce
\begin{align*}
    \bDelta^{(s)} &= \mathrm{diag}(1, 2^s, 3^s, \ldots), \quad  \bDelta^{(s)}_N = \mathrm{diag}(1, 2^s, 3^s, \ldots,N^s).
\end{align*}
We have the following proposition
\begin{proposition}\label{l:multiplication_stability}
    Suppose $\vec f = (f_j)_{j \geq 0}$, $\vec g = (g_j)_{j \geq 0}$ are such that $\vec f ,\vec g \in \ell^1_{s + 1}$ for $s \geq 0$ and set
    \begin{align*}
        f(x) := \sum_{j=0}^\infty f_j p_j(x;0), \quad g(x) := \sum_{j=0}^\infty g_j p_j(x;0).
    \end{align*}
    Then $\bDelta^{(s)} \vec M(f,\lambda)\bDelta^{(-s)}$ and $\bDelta^{(s)} \vec M(g,\lambda)\bDelta^{(-s)}$ are both bounded on $\ell^2(\mathbb N)$ and
    we have
    \begin{align*}
        \|\bDelta^{(s)} (\vec M(f,\lambda) - \vec M( g,\lambda))\bDelta^{(-s)}\|_{\ell^2} \leq C \|\vec f - \vec g\|_{\ell^1_{s+1}}.
    \end{align*}
\end{proposition}
\begin{proof}
    It follows that
    \begin{align*}
        \|T_j(\vec J(\lambda))\|_{\ell^2} \leq 1,  
    \end{align*}
    where $T_j$ is the $j$th Chebyshev polynomial of the first kind. And, in particular, every entry of $T_j(\vec J(\lambda))$ is bounded above by unity, in modulus.  Recall that $p_0(x;0) = T_0(x)$, and $p_j(x;0) = \sqrt{2} T_j(x)$, $j \geq 1$.  Since $T_j(\vec J(\lambda))$ has bandwidth most $j$, let $\vec S_j$ be the semi-infinite matrix with ones on the $j$th diagonal, $\vec S_0 = \id$. We have that
    \begin{align*}
        \|p_j( \bDelta^{(s)}\vec J(\lambda)\bDelta^{(-s)};0)\|_{\ell^2} \leq \sqrt{2} \sum_{\ell = -j}^{-1} (1 - \ell)^s \|\vec S_\ell\|_{\ell^2} + \sqrt{2}\sum_{\ell = 0}^{j} \|\vec S_\ell\|_{\ell^2} \leq \sqrt{2}(j + 1)( 1 + (1+j)^s).
    \end{align*}
    So the series
    \begin{align*}
        \sum_j f_j p_j( \bDelta^{(s)}\vec J(\lambda)\bDelta^{(-s)};0)
    \end{align*}
    is absolutely convergent as a sequence of operators on $\ell^2(\mathbb N)$.  Taking the difference of the two operators and bounding them term-by-term gives the result.
\end{proof}

\begin{corollary}\label{c:bounded}
    Suppose $f \in C^{q,\alpha}(\mathbb I)$ and $ \alpha + q > 2 + s$, then for the Chebyshev first-kind  expansion
    \begin{align*}
       \mathcal I_n^{\rm Ch}f(x): = \sum_{j=0}^{n-1} f_j p_j(x;0), \quad f_j = \langle f, p_j(\diamond;0)\rangle_{\mu_0},
    \end{align*}
    there exists $C > 0$ such that
    \begin{align*}
        \|\bDelta^{(s)} (\vec M(f,\lambda) - \vec M( \mathcal I_n^{\rm Ch}f,\lambda))\bDelta^{(-s)}\|_{\ell^2} \leq C \sum_{j = n}^\infty j^{-k -\alpha + s + 1} = O(n^{-k -\alpha + 2  + s}).
    \end{align*}
    and therefore for another constant $C'$
    \begin{align*}
        \|\bDelta^{(s)} \vec M( \mathcal I_n^{\rm Ch}f,\lambda)\bDelta^{(-s)}\|_{\ell^2} \leq C'.
    \end{align*}    
\end{corollary}
\begin{proof}
      From Jackson's theorem \cite{atkinson}, we can find a polynomial $q_j$ of degree $j-1$ that satisfies $\| f - q_j\|_\infty < D j^{-q - \alpha}$.  Then
    \begin{align*}
        |\langle f, p_j(\diamond;0) \rangle_{\mu_0}|  \leq |\langle q_j, p_j(\diamond;0) \rangle_{\mu_0}| + |\langle f - q_j, p_j(\diamond;0) \rangle_{\mu_0}| \leq D' j^{-q-\alpha},
    \end{align*}
    for a new constant $D'$, and the theorem follows.
\end{proof}

And we have the elementary fact.
\begin{lemma}\label{l:diff_spaces}
The operator $\vec D_j(\lambda) \bDelta^{(-j)}$ is bounded on $\ell^2(\mathbb N)$.
\end{lemma}

\subsection{Comparison of collocation and finite section} \label{sec:comp}
We recall the definition of $\vec F_N$ in \eqref{eq:FN}.
\begin{proposition}
    For $n > 0$, let $P = (x_1,\ldots,x_n)$ be the roots of $p_n(x;\lambda)$ and write
    \begin{align*}
        \vec F_n(\mu_\lambda) \vec P_{\lambda \to P} = \begin{bmatrix} \vec a_1 & \vec a_2 & \cdots \end{bmatrix} = \begin{bmatrix} \id_n & \vec a_{n+1} & \cdots \end{bmatrix}.
    \end{align*}
    That is, $\vec a_j = \vec e_j$ for $j = 1,\ldots,n$. For $j > n$, only the last $n-j$ entries of $\vec a_j$ may be non-zero and $|\vec a_{ij}| \leq C(\lambda)$ where $C(\lambda)$ is the constant in Lemma~\ref{l:aliasing}.  Furthermore, for
    \begin{align*}
        \bDelta_n^{(s)} \begin{bmatrix} \id_n & \vec a_{n+1} & \cdots & \cdots & \vec a_{n + m}\end{bmatrix}  \bDelta_{n+m}^{(-s)}  = \begin{bmatrix} \id_n & \check{\vec a}_{n+1} & \cdots & \cdots & \check{\vec a}_{n + m}\end{bmatrix},
    \end{align*}
    we have 
    \begin{align*}
        \|\check{\vec a}_{n + j}\|_{2}^2 \leq C(\lambda)^2 (n +j)^{-2s}\sum_{i=\max\{n-j +1,1\}}^{n} i^{2s}, \quad 1 \leq j \leq n.
    \end{align*}
\end{proposition}
\begin{proof}
 This is a direct consequence of Lemma~\ref{l:aliasing} and the fact that the Gaussian quadrature rule is exact for polynomials of degree $2n-1$.
\end{proof}

In the previous proposition, for $j \leq n$ we have
\begin{align*}
    \sum_{i=n-j+1}^n i^{2s} \leq  j n^{2s}
\end{align*}
and for $m \leq n$
\begin{align*}
    n^{2s}  \sum_{j=1}^{m} (n +j)^{-2s} j \leq m^2.
\end{align*}
We reach the conclusion that
\begin{align}\label{eq:column-estimate}
    \left\|\begin{bmatrix} \check{\vec a}_{n+1} & \cdots & \cdots & \check{\vec a}_{n + m}\end{bmatrix} \right\|_{\rm F} \leq C(\lambda) m.
\end{align}

We believe something stronger is true:
\begin{conjecture}[Aliasing estimate]\label{conj}
    There exists $c(\lambda) > 0$ such that $\| \vec a_j\|_{\ell^2} \leq c(\lambda)$ for all $n, j$.
\end{conjecture}

\begin{remark}
We note that this conjecture, if true, implies that for $ s > 1/2$
\begin{align*}
    \left\|\begin{bmatrix} \check{\vec a}_{n+1} & \cdots & \cdots & \check{\vec a}_{n + m}\end{bmatrix} \right\|_{\rm F}^2 \leq c(\lambda)^2 n^{2s} \sum_{j=n+1}^{m + n} j^{-2s} = O(n).
\end{align*}
As this implies the Frobenius norm is $O(n^{1/2}) = o(n)$, bounded independent of $m$, it would allow sending $m \to \infty$, for $N$ fixed in Theorem~\ref{t:main}, eliminating the need for some of the extra terms in the proof of Theorem~\ref{t:MAIN}.
\end{remark}

In the entirety of this section, we suppose that the grid $P$ is given by the roots of $p_{N-k}(x;\lambda + k)$ and $a_k(x) \equiv 1$.  The finite-section truncation of \eqref{eq:PGM} is given by
\begin{align}\label{eq:FS}
    \vec L_N^{\rm FS}= \vec L_N^{\rm FS}(a_0,\ldots,a_k) := \vec Q_N^T\begin{bmatrix} \vec S \vec E_{-1}(\lambda,k) + \vec T \vec E_{1}(\lambda,k) \\
    \vec L \end{bmatrix}\vec Q_N \check{\vec u}_N = \begin{bmatrix}
        \vec b\\
        f_0 \\
        \vdots \\
        f_{N-k-1}
    \end{bmatrix}.
\end{align}
We perform a comparison of
\begin{align*}
     {\vec N}_j(a_j) :=\vec Q_N^T\begin{bmatrix} \vec S \vec E_{-1}(\lambda,k) + \vec T \vec E_{1}(\lambda,k) \\
    \vec M(a_j;k + \lambda)\vec C_{j + \lambda \to k + \lambda} \vec D_j(\lambda) \end{bmatrix} \vec Q_N,
    \end{align*}
    and
\begin{align*}
    \begin{bmatrix} \vec S \vec E_{-1}(\lambda,k) + \vec T \vec E_{1}(\lambda,k) \\
    a_j(P)\vec P_{\lambda + j \to P} \vec D_j(\lambda) \end{bmatrix} \vec Q_N.
\end{align*}
But this cannot occur directly as the range of the latter is function values and the former is coefficients. So, instead consider
\begin{align*}
   \tilde {\vec N}_j(a_j) &:=\begin{bmatrix} \id_k & \vec 0\\ \vec 0 & \vec F_{N-k}(\mu_{\lambda + k}) \end{bmatrix} \begin{bmatrix} \vec S \vec E_{-1}(\lambda,k) + \vec T \vec E_{1}(\lambda,k) \\
    a_j(P)\vec P_{\lambda + k \to P} \vec D_j(\lambda) \end{bmatrix} \vec Q_N\\
    & = \begin{bmatrix} \id_k & \vec 0\\ \vec 0 & \vec F_{N-k}(\mu_{\lambda + k}) a_j(P) \vec F_{N-k}(\mu_{\lambda + k})^{-1} \end{bmatrix} \begin{bmatrix} \vec S \vec E_{-1}(\lambda,k) + \vec T \vec E_{1}(\lambda,k) \\
    \vec F_{N-k}(\mu_{\lambda + k})\vec P_{\lambda + k \to P} \vec C_{\lambda +j \to \lambda + k}\vec D_j(\lambda) \end{bmatrix} \vec Q_N.
\end{align*}

We follow the right preconditioning step as in \cite{Olver2013} and define
\begin{align*}
    \vec Z = \begin{bmatrix} \id_k & \vec 0\\ 
    \vec 0 & \vec D_k(\lambda) \end{bmatrix}.
\end{align*}
There exists a constant $c_\lambda > 1$ such that the $n$th diagonal entry $z_{nn}$ of $\vec Z$ satisfies
\begin{align*}
    c_\lambda^{-1} n^k \leq |z_{nn}| \leq c_\lambda n^k, \quad c_\lambda  >0.
\end{align*}
We then set $\vec Z_N$ to be the upper-left $N \times N$ subblock of $\vec Z$.  The next theorem and its corollary are proved in Appendix~\ref{sec:defer}.

\begin{theorem}\label{t:main}
Suppose $f \in C^{q,\alpha}(\mathbb I)$ and $q + \alpha > 2 + s$.
Then for $t \geq 0$ and $ m < N$
\begin{align}
    \left\|\begin{bmatrix} \id_k & \vec 0 \\ \vec 0 & \bDelta^{(s)}_{N-k} \end{bmatrix} (\vec N_j(\mathcal I_m^{\rm Ch} f) - \tilde {\vec N}_j(\mathcal I_m^{\rm Ch} f)) \vec Z_N^{-1} \bDelta_N^{(-s - t)}\right\|_{\ell^2} = O(m (N - m)^{j-k-t}),
\end{align}
with the difference vanishing identically if $f$ is constant and $j = k$.
\end{theorem}


And to state the following corollary, we need to introduce some additional notation.  For $f: \mathbb I \to \mathbb C$ set
\begin{align*}
    \vec f_N = \begin{bmatrix} \langle f, p_0(\diamond;\lambda + k) \rangle_{\mu} \\ \vdots \\ \langle f, p_{N-1}(\diamond;\lambda + k) \rangle_{\mu} \end{bmatrix}, \quad \tilde{\vec f}_N = \begin{bmatrix} \langle f, p_0(\diamond;\lambda + k) \rangle_{\mu,N} \\ \vdots \\ \langle f, p_{N-1}(\diamond;\lambda + k) \rangle_{\mu,N} \end{bmatrix}, \quad \mu = \mu_{k + \lambda}.
\end{align*}

\begin{corollary}\label{c:main}
Suppose that $m = m(N) = o(N)$, $a_j \in C^{q,\alpha}(\mathbb I)$ and $q + \alpha > 2 + s$, $s \geq 0$.  Suppose also that there exists $N_0> 0$, $C > 0$ such that for $N > N_0$, 
$$\vec L_N^{\rm FS} = \vec L_N^{\rm FS}(\mathcal I_m^{\rm Ch} a_0,\ldots,\mathcal I_m^{\rm Ch} a_{k-1}, 1),$$
is invertible and $\|\vec Z_N {\vec L_N^{\rm FS}}^{-1}\|_{\ell^2_s} < C$.  Then for $N$ sufficiently large 
$$\vec L_N^{\rm C}(\mathcal I_m^{\rm Ch} a_0,\ldots,\mathcal I_m^{\rm Ch} a_{k-1}, 1),$$
is invertible, where the collocation nodes are chosen as the roots of $p_{N-k}(x;\lambda + k)$.  If $s$ is sufficiently large\footnote{This can be easily found using Lemma~\ref{l:opgrowth} and in Theorem~\ref{t:MAIN} we will impose more stringent conditions.} so that $\vec E_{\pm 1}(\lambda,k)$ is bounded from $\ell_{s+k}^2$ to $\mathbb C^k$, then the solution $\tilde{\vec u}_N$ of \eqref{eq:coll} satisfies
\begin{align*}
 \|\vec u_N - \tilde{\vec u}_N\|_{\ell^2_{s+k}} = O \left( m N^{-1-t} \|\vec w_N\|_{\ell^2_{t + s}} + \|\vec f_{N-k} - \tilde{\vec f}_{N-k}\|_{\ell^2_s} \right).
\end{align*}
\end{corollary}
\begin{proof}
    We note that
    \begin{align*}
        \|\vec A\|_{\ell_s^2} 
        = \|\bDelta^{(s)}\vec A\bDelta^{(-s)} \vec \|_{\ell^2}.
    \end{align*}
    Set
    \begin{align}\label{eq:tildeL}
       \tilde{\vec L}_N^{\rm FS} = \tilde{\vec L}_N^{\rm FS}(\mathcal I_m^{\rm Ch} a_0,\ldots,\mathcal I_m^{\rm Ch} a_{k-1}, 1) =  \begin{bmatrix} \id_k & \vec 0 \\
        \vec 0 & \vec F_{N-k}(\mu_{\lambda + k})
        \end{bmatrix} \vec L_N^{\rm C}(\mathcal I_m^{\rm Ch} a_0,\ldots,\mathcal I_m^{\rm Ch} a_{k-1}, 1).
    \end{align}
    Then Theorem~\ref{t:main} implies that
    \begin{align*}
        \|(\tilde{\vec L}_N^{\rm FS} - {\vec L}_N^{\rm FS})\vec Z_N^{-1}\|_{\ell^2_s} = O( m N^{k -j}).
    \end{align*}
    This establishes the first claim using Theorem~\ref{t:continuity}.     Then, consider
    \begin{align*}
        \vec K_N = {\vec L}_N^{\rm FS}\vec Z_N^{-1} - \id, \quad \tilde{\vec K}_N = \tilde{\vec L}_N^{\rm FS}\vec Z_N^{-1} - \id,
    \end{align*}
    and the linear systems
    \begin{align}\label{eq:yN}
        {\vec L}_N^{\rm FS}\vec Z_N^{-1} \vec w_N = \begin{bmatrix}
            \vec b \\ \vec f_{N-k} \end{bmatrix} =: \vec y_N, \quad \tilde {\vec L}_N^{\rm FS}\vec Z_N^{-1} \tilde{\vec w}_N = \begin{bmatrix}
            \vec b \\ \tilde{\vec f}_{N-k} \end{bmatrix} =: \tilde{\vec y}_N.
    \end{align}
    Here $\vec w_N = \vec Z_N \vec u_N$, $\tilde{\vec w}_N = \tilde{\vec Z}_N \tilde{\vec u}_N$.  Therefore
    \begin{align*}
        (\id + \vec K_N) \vec w_N &= \vec y_N, \quad 
        (\id + \tilde{\vec K}_N) \tilde{\vec w}_N = \tilde{\vec y}_N,\\
        (\id + \tilde { \vec K}_N) \vec w_N &= (\tilde { \vec K}_N - { \vec K}_N) \vec w_N + \vec y_N.
    \end{align*}
    Thus
    \begin{align*}
    \vec w_N - \tilde{\vec w}_N = (\id + \tilde { \vec K}_N)^{-1}\left[(\tilde { \vec K}_N - { \vec K}_N) \vec w_N +  \vec y_N - \tilde{\vec y}_N \right].
    \end{align*}
    And therefore, for $N$ sufficiently large
    \begin{align*}
        \|\vec Z_N(\vec u_N - \tilde{\vec u}_N)\|_{\ell^2_s} \leq 2C \left( \|(\tilde { \vec K}_N - { \vec K}_N) \vec w_N\|_{\ell_s^2} + C_s \|\vec f_{N-k} - \tilde{\vec f}_{N-k}\|_{\ell^2_s} \right).
    \end{align*}
    Then we estimate, again using Theorem~\ref{t:main},
    \begin{align*}
        \|(\tilde { \vec K}_N - { \vec K}_N) \vec w_N\|_{\ell_s^2} \leq D' m (N - m)^{-1} N^{-t} \|\vec w_N\|_{\ell_{s+t}^2} = D' m (N - m)^{-1} N^{-t} \|\vec Z_N \vec u_N\|_{\ell_{s+t}^2}.
    \end{align*}
\end{proof}


\subsection{Stability estimates for the ultraspherical method}\label{sec:conv-usm}

This section is concerned with how finite-section truncations of \eqref{eq:PGM} converge to the true solution.  Here, following Olver \& Townsend, after right preconditioning, abstract theory can be applied.  Consider
\begin{align*}
     \vec L : =  \vec D_k(\lambda) + \sum_{j=0}^{k-1} \vec M(a_j;k + \lambda)\vec C_{j + \lambda \to k + \lambda} \vec D_j(\lambda).
\end{align*}
Then, we write \eqref{eq:PGM} using $\vec w = \vec Z \vec u$ and define $\vec K = \vec K(a_0,\ldots,a_{k-1})$ by
\begin{align}
    (\id + \vec K) \vec w = \left( \id + \begin{bmatrix} [\vec S \vec E_{-1}(\lambda,k) + \vec T \vec E_{1}(\lambda,k) - \id_k] \vec Z^{-1} \\
    \displaystyle \sum_{j=0}^{k-1} \vec M(a_j;k + \lambda)\vec C_{j + \lambda \to k + \lambda} \vec D_j(\lambda)\vec Z^{-1} \end{bmatrix}\right) \vec w = \begin{bmatrix}
        \vec b\\
        \vec C_{\lambda \to \lambda +k} \vec f
    \end{bmatrix}.
\end{align}
So, we focus on operators
\begin{align*}
\vec M(a_j;k + \lambda)\vec C_{j + \lambda \to k + \lambda} \vec D_j(\lambda)\vec Z^{-1}.
\end{align*}
We see that $\vec D_j(\lambda)\vec Z^{-1}$ is bounded from $\ell_s^2(\mathbb N)$ to $\ell_{s +k - j}^2(\mathbb N)$.  And we use the following
\begin{lemma}
    For $t > s$, $\ell^2_t(\mathbb N)$ is compactly embedded in $\ell^2_s(\mathbb N)$. 
\end{lemma}
\begin{proof}
    Suppose $(u_j)_{j=1}^\infty = \vec u \in \ell^2_t(\mathbb N)$.  Then
    \begin{align*}
        \sum_{j=n}^\infty j^{2s} |u_j|^2 = \sum_{j=n}^\infty j^{2(s-t)} j^{2t}|u_j|^2 \leq n^{2(s-t)} \|\vec u\|_{\ell_t^2}^2.
    \end{align*}
    Thus the identity $\id: \ell^2_t(\mathbb N) \to \ell^2_s(\mathbb N)$ can be approximated by finite-dimensional (compact) projections in operator norm.  This proves the claim.
\end{proof}

The proof of the following can be found in Appendix~\ref{sec:defer}.
\begin{theorem}\label{t:inf}
Supposing $a_k = 1$, the following hold:
\begin{enumerate}
    \item If for $j = 0,1,\ldots,k-1$, $a_j \in C^{q,\alpha}(\mathbb I)$, $\alpha + q > 2 + s$, then the operator $\vec K$ is compact on $\ell^2_s(\mathbb N)$.  
    \item Suppose the boundary-value problem \eqref{eq:bvp} is uniquely solvable, $\alpha + q > 2 + s$, and $s > \lambda + k + 1/2$ then $\id + \vec K$ is invertible on $\ell_s^2(\mathbb R)$.  
    \item Given the assumptions of (2), there exists $N_0 > 0$ such that if $N > N_0$ then
    \begin{align*}
        \| \vec Z_N {\vec L_{N}^{\rm FS}}^{-1}\|_{\ell^2_s} \leq 2 \|(\id + \vec K)^{-1}\|_{\ell^2_s}.
    \end{align*}
    \item Given the assumptions of (2), there exists $c,C> 0$ such that the solution $\check {\vec u}_N$ of \eqref{eq:FS} and the solution $\vec u$ of \eqref{eq:PGM} satisfy
    \begin{align*}
        c\|\vec Z(\vec u  - \vec Q_N \vec Q_N^T \vec u) \|_{\ell^2_s} \leq \|\vec Z(\vec u  - \vec Q_N \check {\vec u}_N)\|_{\ell^2_s} \leq C \|\vec Z(\vec u  - \vec Q_N \vec Q_N^T \vec u) \|_{\ell^2_s}
    \end{align*}
    for $N > N_0$.
\end{enumerate}
\end{theorem}

While this proves convergence of the finite section method applied to \eqref{eq:PGM}, this method is, in principle, unimplementable because the operators $\vec M(a_j,\lambda)$ cannot be computed exactly unless $a_j$ is a polynomial.  So, we now prove a straightforward stability lemma about the replacement of these functions with polynomial approximations. It is a direct consequence of Lemma~\ref{l:multiplication_stability}.

\begin{lemma}\label{l:ultra_stable}
    Suppose
    \begin{align*}
        a_j(x) = \sum_{i = 0}^\infty a_{j,i} p_j(x;0), \quad \tilde a_j(x) = \sum_{i = 0}^\infty \tilde a_{j,i} p_j(x;0), \quad j = 0,1,\ldots, k-1,
    \end{align*}
    for coefficients satisfying $\vec a_j = (a_{j,i})_{i \geq 0}$, $\tilde{\vec a}_j = (\tilde a_{j,i})_{i \geq 0}$, $\|\vec a_j - \tilde{\vec a}_j\|_{\ell^1_{s+1}} < \epsilon$,
    then there exists $C > 0$ such that
    \begin{align*}
        \|\vec K(\tilde a_0,\ldots,\tilde a_{k-1}) - \vec K(a_0,\ldots,a_{k-1}) \|_{\ell^2_s} < C \epsilon.
    \end{align*}
\end{lemma}

\subsection{Stability estimates for the collocation method}\label{sec:stab-coll}

The last piece of the theory to prove convergence of the collocation method is to, at the level of collocation, establish how small perturbations in the coefficient functions $a_j$ can affect the norm of the resulting linear system.   The following is proved in Appendix~\ref{sec:defer}.

\begin{proposition}\label{p:coll_stable}
    Let $\tilde {\vec L}_N^{\rm FS}$ be as in \eqref{eq:tildeL} and suppose $s  > k + \lambda + 1/2$.  Then there exists a constant $C_{k,\lambda,s}$ such that
    \begin{align*}
        \|(\tilde {\vec L}_N^{\rm FS}(a_0,\ldots,a_{k-1},1) - \tilde {\vec L}_N^{\rm FS}(\tilde a_0,\ldots,\tilde a_{k-1},1)) \vec Z_N^{-1} \|_{\ell_s^2} \leq C_{k,\lambda,s} \max_j \|a_j - \tilde a_j\|_\infty N^s.
    \end{align*}
\end{proposition}

In applying the previous proposition, we note that there is a restriction from Theorem~\ref{t:inf} that $\alpha + q > 2 +s$.  Classical results imply (see \cite{atkinson}, for example) that, for $m > 1$
\begin{align*}
    \|a_j - \mathcal I_m^{\rm Ch} a_j\|_\infty \leq C \frac{\log m}{m^{q + \alpha}}.
\end{align*}
For the bound in the previous proposition will need to tend to zero, while maintaining $m \ll N$, from Corollary~\ref{c:main}, if suffices to take $m  = \lfloor N^\gamma \rfloor$, $\gamma = s/(2 + s)$.

\subsection{The main theorem}\label{sec:main}

The theorem that follows is the main result of this paper.  The constants involved can surely be optimized beyond what is presented here.  Some constants are kept to show the reader that (1) only a finite amount of smoothness of the coefficient functions is required for convergence and (2) how an infinite amount of smoothness results in beyond-all-orders, or spectral, convergence, see Corollary~\ref{c:MAIN}.

\begin{theorem}\label{t:MAIN}
Suppose the following hold:
\begin{enumerate}
    \item $s > \lambda + k + 1/2$,
    \item $a_k = 1$ in \eqref{eq:bvp} and the boundary-value problem \eqref{eq:bvp} is uniquely solvable, and
    \item  $f \in C^{q,\alpha}(\mathbb I)$,   $a_j \in C^{q,\alpha}(\mathbb I)$, $j = 0,1,\ldots,k-1$, $\alpha + q > 2 + s + t$, $t \geq 0$.
\end{enumerate}
Then with $m = \lfloor N^{s/(2 + s)} \rfloor$
\begin{align*}
    \|\vec u - \vec Q_N \tilde {\vec u}_N\|_{\ell^2_{s+k}} &= O\left( \|\vec Z(\vec u  - \vec Q_N \vec Q_N^T \vec u) \|_{\ell^2_s} + N^s\max_{j}\|a_j - \mathcal I_m^{\rm Ch} a_j\|_\infty\|\tilde{\vec y}_{N}\|_{\ell^2_{s}} \right.\\
    & \left.+ \max_j\|\vec a_j - \tilde{\vec a}_j\|_{\ell^1_{s + 1}}\|{\vec y}\|_{\ell^2_{s}} + m N^{-1 -t}\|{\vec y}\|_{\ell^2_{s+t}} + \|\vec f_{N-k} - \tilde{\vec f}_{N-k}\|_{\ell^2_s} \right),
\end{align*}
where
\begin{align*}
    {\vec a}_j = \begin{bmatrix} \langle a_j, p_0(\diamond;0) \rangle_{\mu_{0}} \\ 
     \langle a_j, p_1(\diamond;0) \rangle_{\mu_{0}}\\
     \vdots\\ \end{bmatrix}, \quad \tilde {\vec a}_j = \vec Q_m \vec Q_m^T \vec a_j.
\end{align*}
\end{theorem}
\begin{proof}
    We need to define a number of solutions of linear systems:
    \begin{enumerate}
        \item $\vec u$ is the solution of the full, infinite linear system \eqref{eq:PGM}.
        \item $\tilde {\vec u}_N$ is the solution of \eqref{eq:coll}.
        \item $\check {\vec u}_N$ is the solution of the finite-section system \eqref{eq:FS}.
        \item $\check {\vec u}_{N,m}$ is the solution of \eqref{eq:FS} with $a_j$ replaced with $\mathcal I_m^{\rm Ch} a_j$ for all $j$.
        \item $\tilde {\vec u}_{N,m}$ is the solution of \eqref{eq:coll} with $a_j$ replaced with $\mathcal I_m^{\rm Ch} a_j$ for all $j$.
    \end{enumerate}
    Let us first settle the fact that these quantities are all well-defined for $N$ sufficiently large: (1) is well-defined by Theorem~\ref{t:inf}(2) and (3) is well-defined by Theorem~\ref{t:inf}(3).   Then applying Lemma~\ref{l:ultra_stable}, we see that because $\alpha + q > 2 + s$, we can also use Corollary~\ref{c:bounded} provided that $m \to \infty$.  Specifically, we choose $m = \lfloor N^{s/(2 + s)} \rfloor$.  Thus, $\check {\vec u}_{N,m}$ is well-defined. And this establishes the uniform bound needed in Corollary~\ref{c:main} that then shows  $\tilde {\vec u}_{N,m}$ is well-defined. 
    
We use the sequence of approximations as follows
\begin{align*}
     \|\vec Z(\vec u - \vec Q_N \tilde {\vec u}_N)\|_{\ell^2_s} & \leq \|\vec Z(\vec u - \vec Q_N \check {\vec u}_N)\|_{\ell^2_s} + \|\vec Z_N(\check {\vec u}_{N} - \check {\vec u}_{N,m})\|_{\ell^2_s} + \|\vec Z_N(\check {\vec u}_{N,m} - \tilde {\vec u}_{N,m})\|_{\ell^2_s} \\
     & + \|\vec Z_N(\tilde  {\vec u}_{N,m} - \tilde {\vec u}_{N})\|_{\ell^2_s}.
\end{align*}
And we bound each term individually, for sufficiently large $N$:
\begin{align*}
   \|\vec Z(\vec u - \vec Q_N \check {\vec u}_N)\|_{\ell^2_s} &\leq C \|\vec Z(\vec u  - \vec Q_N \vec Q_N^T \vec u) \|_{\ell^2_s}, \quad (\text{Theorem~\ref{t:inf}}),\\
   \|\vec Z_N(\check {\vec u}_{N} - \check {\vec u}_{N,m})\|_{\ell^2_s} &\leq C \max_j \|\vec a_j - \tilde{\vec a}_j\|_{\ell^1_{s+1}}\|\vec y_{N}\|_{\ell^2_s}, \quad (\text{Lemma~\ref{l:ultra_stable}}),\\
   \|\vec Z_N(\check {\vec u}_{N,m} - \tilde {\vec u}_{N,m})\|_{\ell^2_s} & \leq C \left( m N^{-1-t} \|\vec Z_N \check {\vec u}_{N,m}\|_{\ell^2_{t + s}} + \|\vec f_{N-k} - \tilde{\vec f}_{N-k}\|_{\ell^2_s} \right), \quad (\text{Corollary~\ref{c:main}}),\\
   \|\vec Z_N(\tilde  {\vec u}_{N,m} - \tilde {\vec u}_{N})\|_{\ell^2_s} &\leq C N^s\max_{j}\|a_j - \mathcal I_m^{\rm Ch} a_j\|_\infty \|\tilde{\vec y}_N\|_{\ell^2_s}, \quad (\text{Proposition~\ref{p:coll_stable}}).
\end{align*}
It remains to find a uniform estimate for  $\|\vec Z_N \check {\vec u}_{N,m}\|_{\ell^2_{t + s}}$.  To do this, we need to be able to repeat the first two estimates with $s$ replaced with $s + t$ to obtain
\begin{align*}
    \|\vec Z_N \check {\vec u}_{N,m}\|_{\ell^2_{t + s}} &\leq C \max_j \|\vec a_j - \tilde{\vec a}_j\|_{\ell^1_{s+t+1}}\|\vec y_{N}\|_{\ell^2_{s+t}} + \|\vec Z_N \check {\vec u}_N\|_{\ell_{s+t}^2} \\
    &\leq C'\left[  \max_j \|\vec a_j - \tilde{\vec a}_j\|_{\ell^1_{s+t+1}}\|\vec y_{N}\|_{\ell^2_{s+t}} + \|\vec Z \vec u\|_{\ell_{s+t}^2} + \|\vec Z(\vec u  - \vec Q_N \vec Q_N^T \vec u) \|_{\ell^2_{s+t}} \right].
\end{align*}
This right-hand side is finite, and uniformly bounded in $N$ by a constant times $\|\vec y\|_{\ell^2_{s + t}}$.  Lastly, we note that multiplication by $\vec Z$ gives a norm equivalent to $\ell^2_{s + k}$. 
\end{proof}

\begin{corollary}\label{c:MAIN}
    Suppose $a_k = 1$ and $a_j \in C^\infty(\mathbb T)$ for $j = 0,\ldots,k-1$, $f \in C^\infty(\mathbb T)$.  Then for every $t > 0$
    \begin{align*}
        \|\vec u - \vec Q_N \tilde {\vec u}_N\|_{\ell^2_{s + k}} \leq C_t N^{-t},
    \end{align*}
    for some constant $C_t >0$.
\end{corollary}
\begin{proof}
    Following the proof of Corollary~\ref{c:bounded}
    \begin{align*}
        \max_j \|\vec a_j - \tilde{\vec a}_j\|_{\ell^1_{s+1}} &= O(m^{-q - \alpha + 2 + s}).
        \end{align*}
    From Jackson's theorem and the Lebesgue constant for $\mathcal I_m^{\rm Ch}$:
        \begin{align*}
        \max_{j}\|a_j - \mathcal I_m^{\rm Ch} a_j\|_\infty &= O( m^{-q - \alpha} \log m).
        \end{align*}
    Then we write, $\mu = \mu_{\lambda + k}$, $f_j =   \langle f, p_j(x;\lambda +k) \rangle_{\mu}$, $\tilde f_j =   \langle f, p_j(x;\lambda +k) \rangle_{\mu,N-k}$ giving
    \begin{align*}
       \tilde f_j = f_j + \sum_{\ell=N-k+1}^\infty f_\ell \langle p_\ell(\diamond;\lambda + k), p_j(\diamond;\lambda +k) \rangle_{\mu,N},
    \end{align*}
    provided this sum converges.   And more generally, we estimate, supposing $t > 1/2$, by Lemma~\ref{l:aliasing},
    \begin{align*}
        \|\vec f_{N-k} - \tilde{\vec f}_{N-k} \|_{\ell^2_s}^2  &= \sum_{j=0}^{N -k -1} |f_j - \tilde f_j|^2(j +1)^{2s} \leq C(\lambda)^2 \sum_{j=0}^{N -k -1} \left|\sum_{\ell=N-k+1}^\infty f_\ell \right|^2(j +1)^{2s}\\
        & = C(\lambda)^2 \|\vec f\|_{\ell^2_t}^2 \sum_{j=0}^{N -k -1} \left[\sum_{\ell=N-k+1}^\infty (\ell + 1)^{-2t}\right](j +1)^{2s}\\
        & \leq D \|\vec f\|_{\ell^2_t}^2 N^{2s - 2t + 2},
    \end{align*}
    for a new constant $D$ depending on $t, \lambda$. This gives
    \begin{align*}
        \|\vec f_{N-k} - \tilde{\vec f}_{N-k} \|_{\ell^2_s} \leq D \|\vec f\|_{\ell^2_t} N^{s - t + 1}.
    \end{align*}
    And for $\|\vec f\|_{\ell^2_t}$ to be finite, going back to the proof of Corollary~\ref{c:bounded},  it suffices to impose that $\alpha + q  > t + 1/2$.  This also shows that $\|\vec y\|_{\ell^2_t}$ is finite. Then we note that $\vec u \in \ell^2_{s + t}(\mathbb N)$ for every $t > 0$ because Theorem~\ref{t:inf}(2) applies with $s$ replaced with $s + t$.
\end{proof}

\section{Numerical demonstration}\label{sec:demo}

We now solve some specific differential equations to demonstrate the URC method's effectiveness. But first, we discuss the methodology for estimating errors.  As above, $N$ is the size of the linear system.  The system is solved giving the approximate coefficients.  A grid of equally spaced points is selected on $\mathbb I$ and Clenshaw's algorithm is used to evaluate the orthogonal polynomial series on the grid.  The maximum difference of these values and a reference solution evaluated on this grid determine the error.  In most cases below, the reference solution is the true solution as it can be determined explicitly.  

The three choices of collocation nodes we consider below are:
\begin{itemize}[align=left,leftmargin=*]
    \item[\textbf{First-kind zeros:}] These are the zeros of the Chebyshev polynomials of the first kind
    \begin{align*}
        x_j = \cos \left( \frac{2j -1}{2N} \pi \right), \quad j = 1,2,\ldots,N.
    \end{align*}
    \item[\textbf{First-kind extrema:}] These are the extrema of the Chebyshev polynomials of the first kind
    \begin{align*}
        x_j = \cos \left( \frac{j -1}{N-1} \pi \right), \quad j = 1,2,\ldots,N.
    \end{align*}
    \item[\textbf{Ultraspherical zeros:}] Given a $k$th order differential operator and $\lambda \geq 0$, $(x_j)_{j=1}^N$ are the roots of $p_N(x;k + \lambda)$.
\end{itemize}
Most of our computations are performed with $\lambda = 0$ as this seems to perform the best in practice, see the bottom panel of Figure~\ref{fig:airy}.  Recall that Theorem~\ref{t:restate} (Theorem~\ref{t:Kras}) applies to the first two choices and Theorem~\ref{t:informal} (Theorem~\ref{t:main}) applies to the last choice.

\subsection{Convergence and the choice of nodes}\label{sec:nodes}

\subsubsection{Example 1}
Consider the boundary-value problem
\begin{align}\label{eq:second-order}
    -\frac{d^2 u}{dx^2} - 25 u = 0, \quad u(-1) = 1, \quad u(1) = -1.
\end{align}
Clearly, $u(x) =  - \csc(5) \sin(5 x)$.  The convergence of the URC method for the three choices of collocation points is shown in Figure~\ref{fig:second-order}.  All choices perform well, with the ultraspherical zeros performing slightly better.
\begin{figure}[tbp]
    \centering
    \includegraphics[width=.5\linewidth]{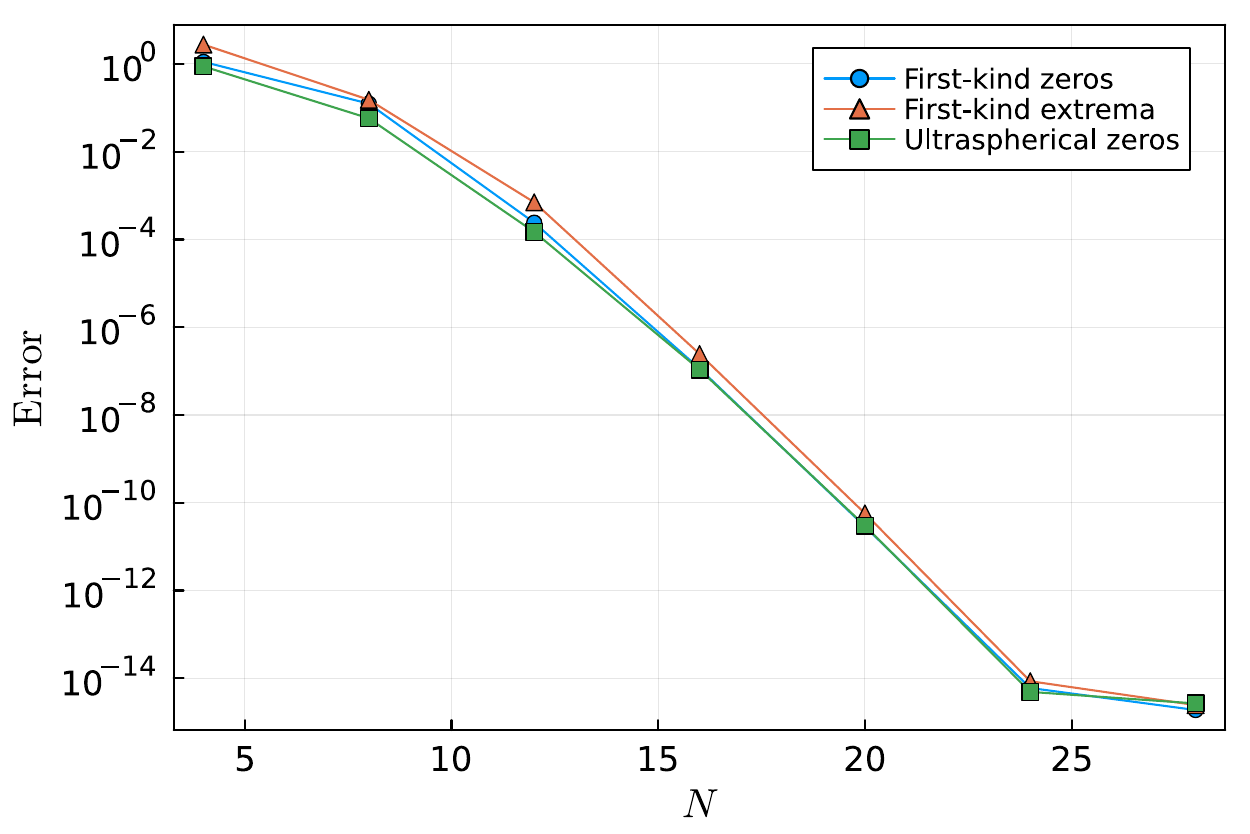}
    \caption{The convergence of the URC method applied to \eqref{eq:second-order}.}
    \label{fig:second-order}
\end{figure}

\subsubsection{Example 2}
Consider the boundary-value problem
\begin{align}\label{eq:third-order}
    -\frac{d^3 u}{dx^3} - 10000 x u = 0, \quad u(-1) = 1, \quad u(1) = -1, \quad u'(-1) = 0.
\end{align}
Here, we do not use an explicit solution, but we compute a reference solution with $N = 500$. The convergence of the URC method for the three choices of collocation points is shown in Figure~\ref{fig:third-order}.  All three choices perform well again, with the ultraspherical zeros performing slightly better initially and not as well in the intermediate regime.
\begin{figure}[tbp]
    \centering
    \begin{subfigure}[b]{0.48\textwidth}
    \centering
    \includegraphics[width=\linewidth]{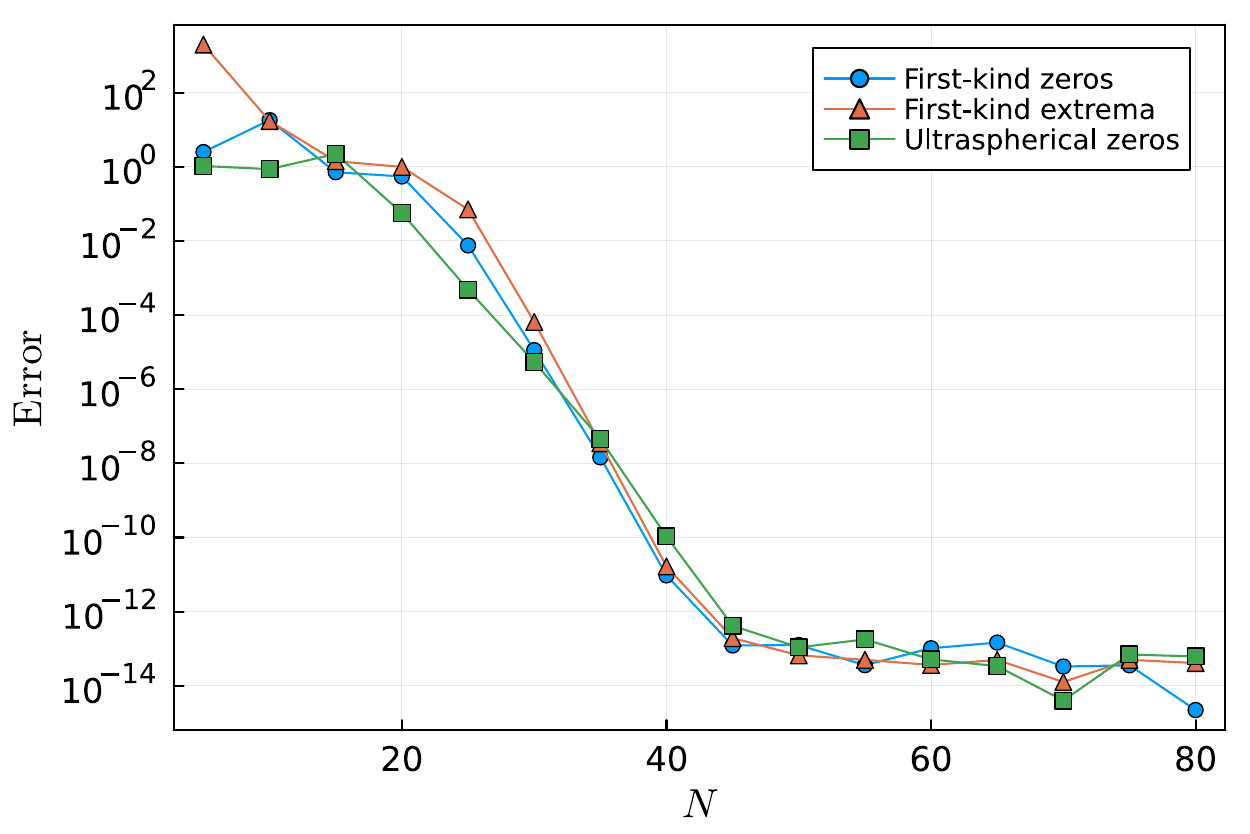}
    \caption{}
    \end{subfigure}
    \begin{subfigure}[b]{0.48\textwidth}
    \centering
    \includegraphics[width=\linewidth]{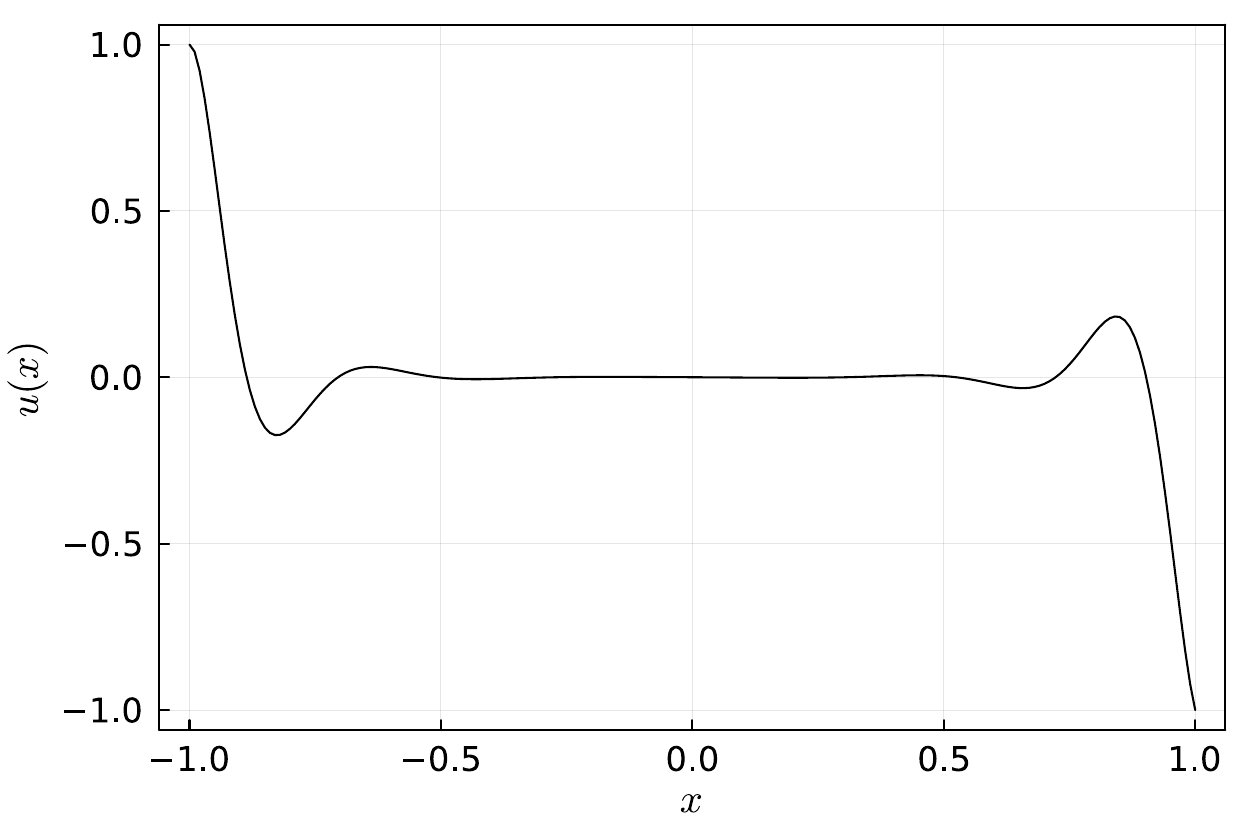}
    \caption{}
    \end{subfigure}
    \caption{(A) The convergence of the URC method applied to \eqref{eq:third-order}. (B) The solution with $N = 500$.}
    \label{fig:third-order}
\end{figure}

\subsubsection{Example 3}
Consider the boundary-value problem
\begin{align}\label{eq:boundary-layer}
    \epsilon \frac{d^2 u}{dx^2} +\frac{d u}{dx} +  u = 0, \quad u(-1) = 0, \quad u(1) = 1, \quad \epsilon = 10^{-3}.
\end{align}
The solution exhibits a boundary layer at $x = -1$.  Here, again, while we could, we do not use an explicit solution, but we compute a reference solution with $N = 1000$. The convergence of the URC method for the three choices of collocation points is shown in Figure~\ref{fig:boundary-layer}.  All three choices again perform well, with the first-kind extrema performing the worst.
\begin{figure}[tbp]
    \centering
    \begin{subfigure}[b]{0.48\textwidth}
    \centering
    \includegraphics[width=\linewidth]{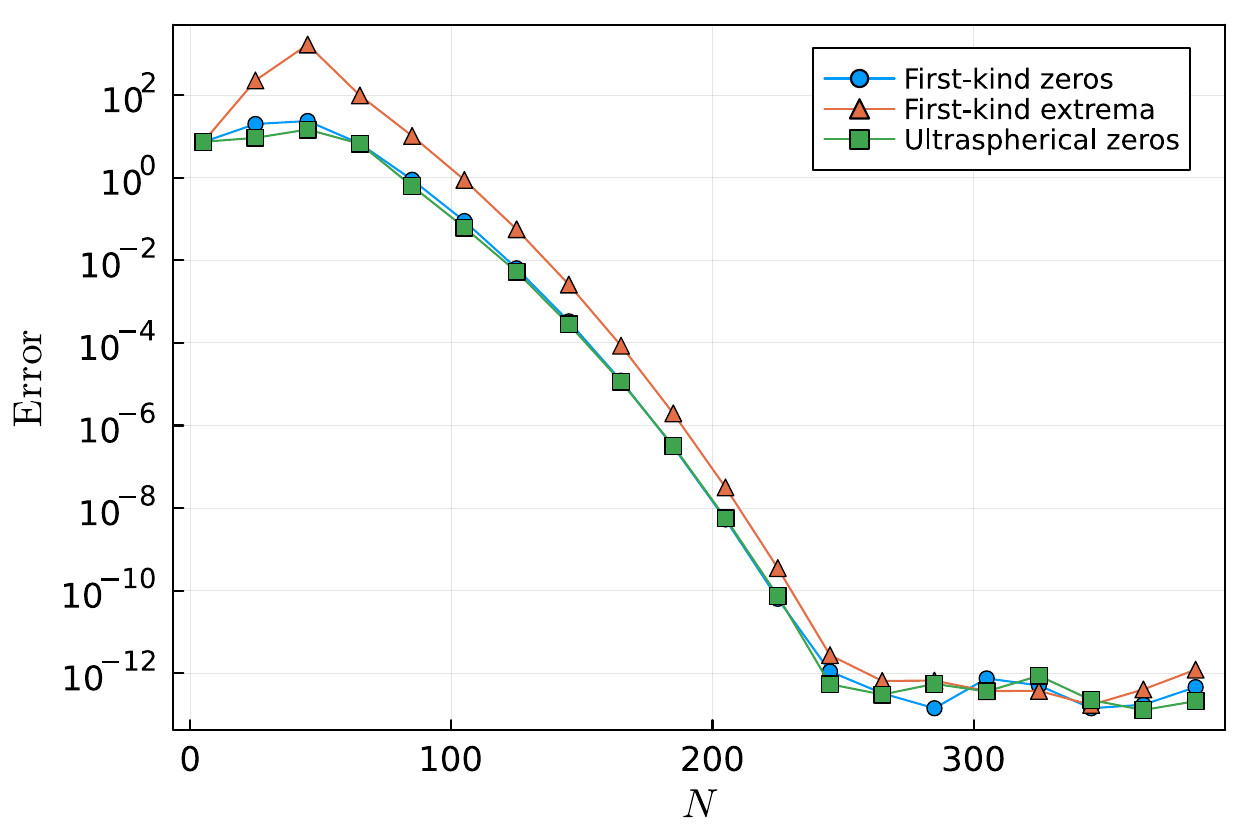}
    \caption{}
    \end{subfigure}
    \begin{subfigure}[b]{0.48\textwidth}
    \centering
    \includegraphics[width=\linewidth]{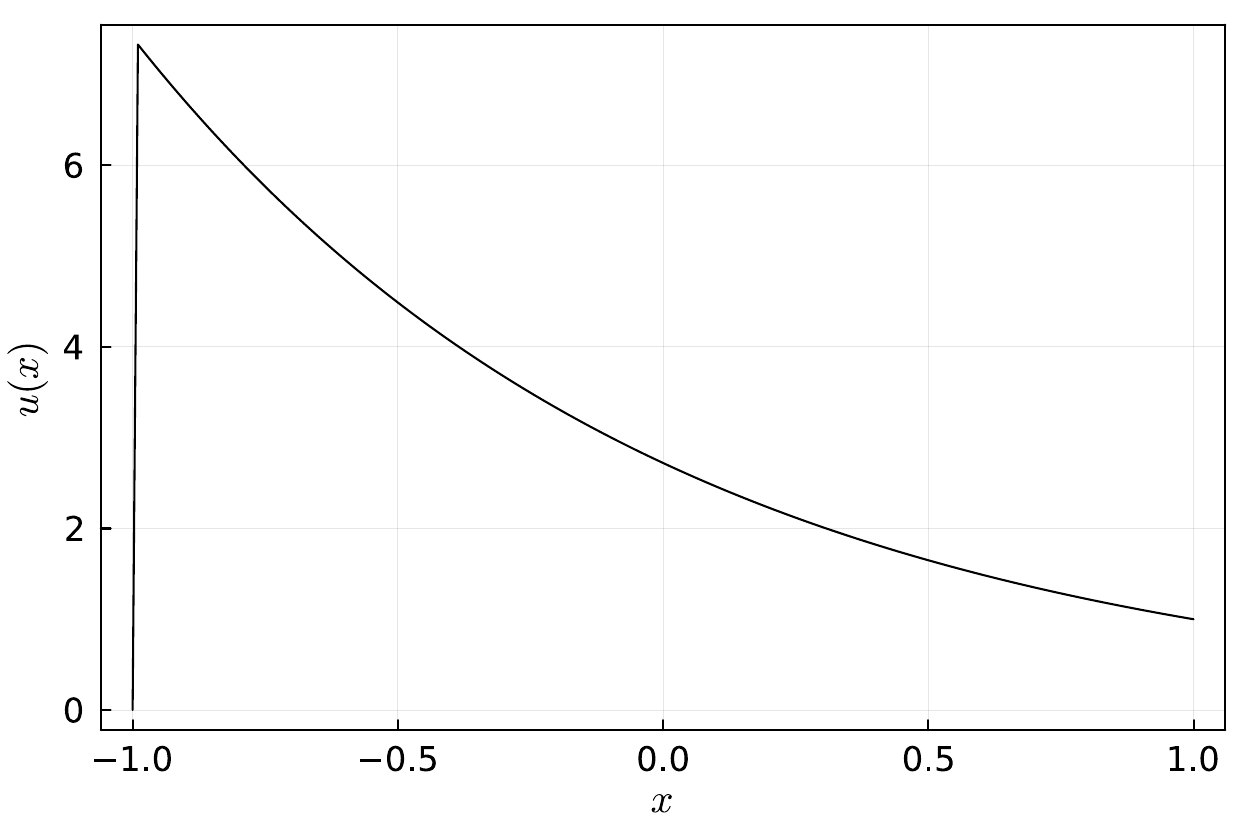}
    \caption{}
    \end{subfigure}
    \caption{(A) The convergence of the URC method applied to \eqref{eq:boundary-layer}. (B) The solution with $N = 1000$.}
    \label{fig:boundary-layer}
\end{figure}

\subsubsection{Example 4}
Consider the boundary-value problem
\begin{align}\label{eq:airy}
    \epsilon^3 \frac{d^2 u}{dx^2} - x  u = 0, \quad u(-1)={\rm Ai}(-1/\epsilon), \quad u(1) = {\rm Ai}(1/\epsilon), \quad \epsilon = 10^{-2}.
\end{align}
The solution is given by $u(x) = {\rm Ai}(x/\epsilon)$ where ${\rm Ai}$ is the Airy function \cite{DLMF}. The convergence of the URC method for the three choices of collocation points is shown in Figure~\ref{fig:airy}.  All three choices again perform well, with the ultraspherical zeros peforming the best.  We also see that $\lambda = 0$ is preferable.
\begin{figure}[tbp]
    \centering

    \begin{subfigure}[b]{0.48\textwidth}
    \centering
    \includegraphics[width=\linewidth]{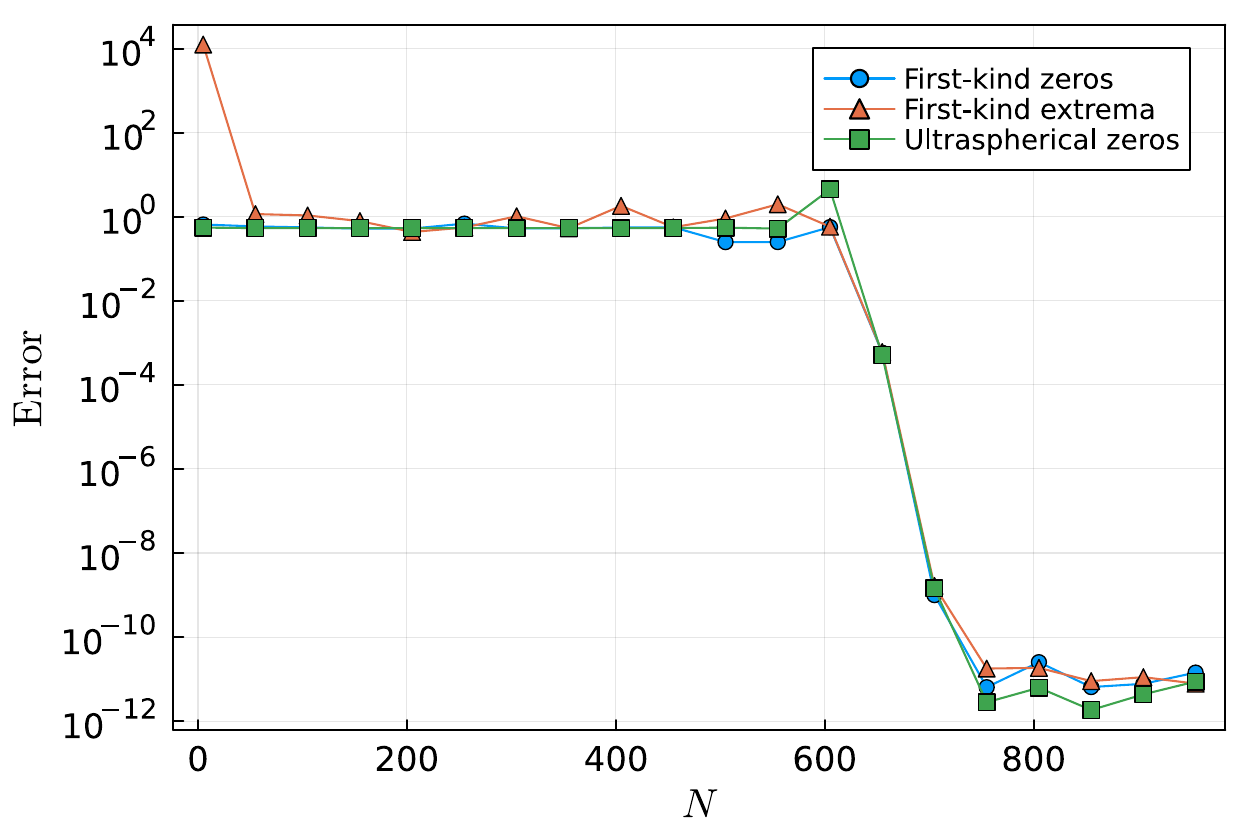}
    \caption{}
    \end{subfigure}
    \begin{subfigure}[b]{0.48\textwidth}
    \centering
    \includegraphics[width=\linewidth]{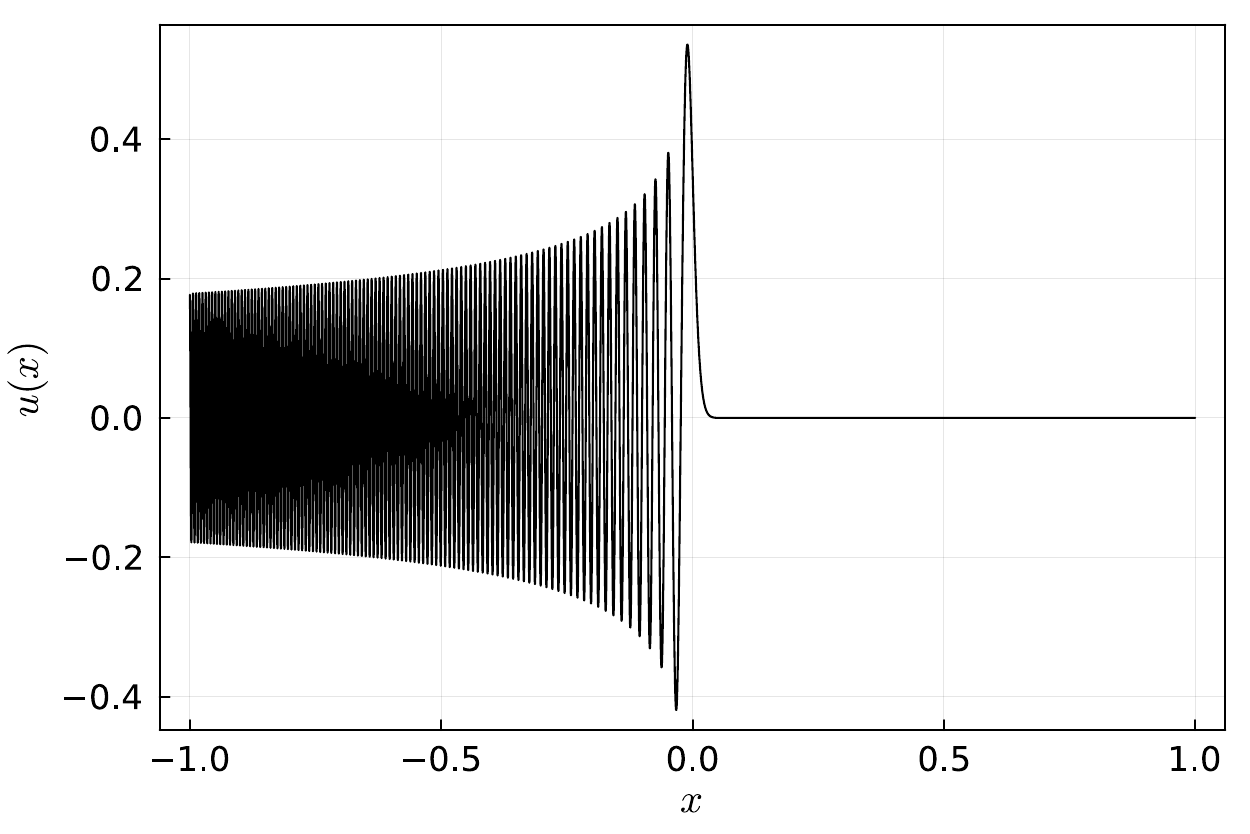}
    \caption{}
    \end{subfigure}
    \begin{subfigure}[b]{0.48\textwidth}
    \centering
    \includegraphics[width=\linewidth]{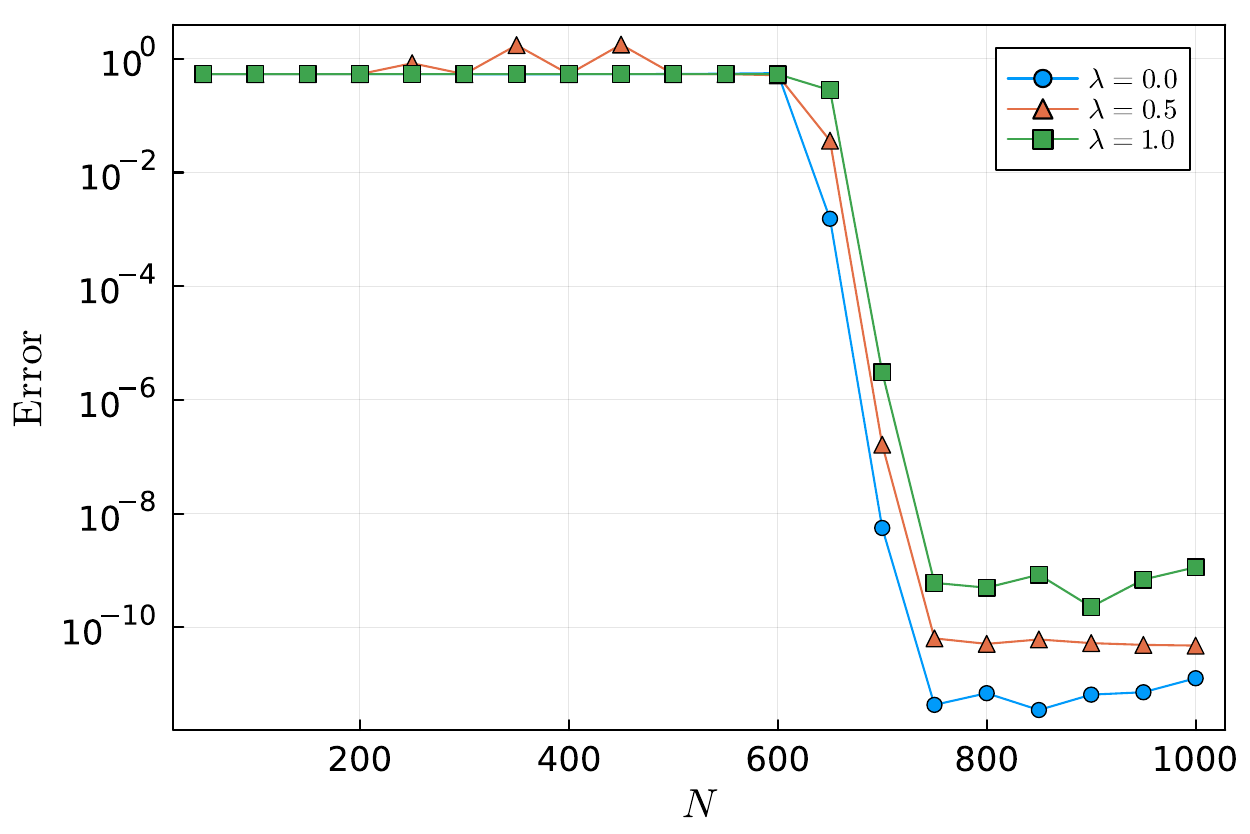}
    \caption{}
    \end{subfigure}
    \caption{(A) The convergence of the URC method applied to \eqref{eq:boundary-layer}. (B) The true solution. (C)  The effect of varying $\lambda = 0,1/2,1$.}
    \label{fig:airy}
\end{figure}

\subsubsection{Example 5} As a last example, we consider a boundary-value problem with non-smooth coefficients that has a smooth solution.  Specifically, consider
\begin{align}\label{eq:abs}
    \frac{d^2 u}{dx^2} + |x| u(x) = \left(|x|-\frac{\pi^2}{4}\right)\sin(\pi x /2), \quad u(-1) = -1, \quad u(1) = 1.
\end{align}
It follows that $u(x) = \sin(\pi x /2)$.  While it is still possible to implement it, the use of the approach of Olver \& Townsend would be more challenging here because the orthogonal polynomial expansion of the absolute value function converges very slowly.  Nevertheless, due to the optimality elucidated in Theorem~\ref{t:restate} (Theorem~\ref{t:Kras}), the method converges very fast, see Figure~\ref{fig:abs}.  This indicates that Theorem~\ref{t:informal} (Theorem~\ref{t:MAIN}) is likely pessimistic.
\begin{figure}[tbp]
    \centering
    \includegraphics[width=.48\linewidth]{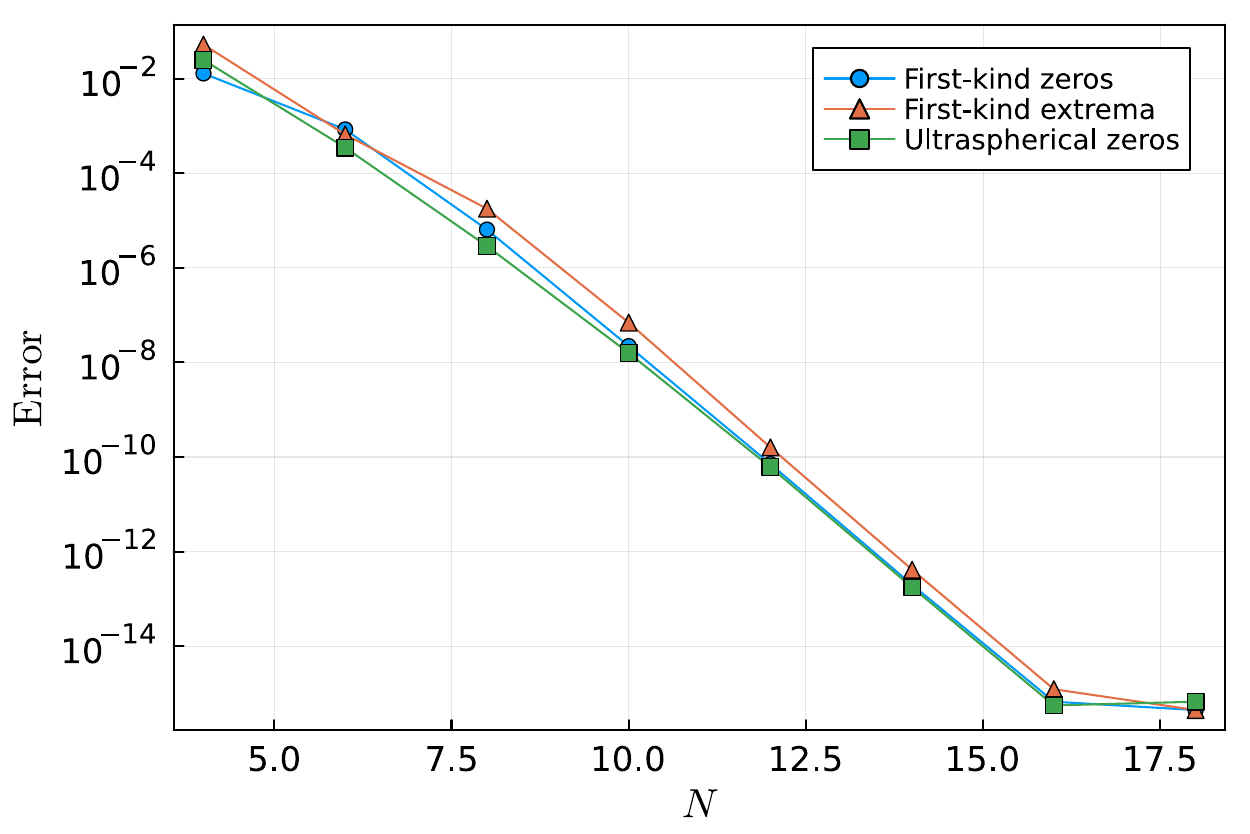}
    \caption{The convergence of the URC method applied to \eqref{eq:abs} demonstrating that the convergence rate is determined by the smoothness of the true solution and not by the smoothness of the coefficient functions.}
    \label{fig:abs}
\end{figure}

\subsection{Preconditioning for GMRES}\label{sec:precond}

In this section, we consider the iterative solution of the collocation system.  See \cite{Du2015}, for example, for discussion of preconditioning the method of Driscoll \& Hale.  Our approach is more straightforward and does not require so-called Birkhoff interpolation.  If the original system is
\begin{align*}
    \vec L_N^{\rm C} \tilde{\vec u}_N = \vec b_N,
\end{align*}
we recast it as
\begin{align*}
    \bDelta_N^{(s)}\begin{bmatrix} \id_k & \vec 0 \\
    \vec 0 & \vec F_{N-k}(\mu_{\lambda +k}) \end{bmatrix}\vec L_N^{\rm C}\vec Z_N^{-1} \bDelta_N^{(-s)} \tilde{\vec v}_N =  \bDelta_N^{(s)}\begin{bmatrix} \id_k & \vec 0 \\
    \vec 0 & \vec F_{N-k}(\mu_{\lambda +k}) \end{bmatrix} \vec b_N, \quad \tilde{\vec u}_N = \bDelta_N^{(-s)} \vec Z_N^{-1} \tilde{\vec v}_N.
\end{align*}
This gives a diagonal, right-preconditioner and dense, but easily computable, inverse-free left preconditioner.  Choosing $s$ is important, and it is really informed by the growth along the columns of the matrices $\vec E_{\pm 1}(\lambda, k)$.  Specifically, the row vector
\begin{align*}
\vec P_{\ell + \lambda \to \{\pm 1\}} \vec D_\ell(\lambda)
\end{align*}
grows as $j^{2\ell + \lambda}$ where $j$ is the column index.  Multiplication on the right by $\vec Z^{-1}_N$  will effectively compensate by a factor of $j^{-k}$.  So we could choose  the smallest value of $s \geq 0$ so that
\begin{align*}
    2\ell + \lambda -k - s \leq -1,
\end{align*}
and thus this row vector will correspond to a uniformly bounded linear functional\footnote{Technically, it suffices to have $2\ell + \lambda -k - s < -1/2$ but in the examples explored here, choosing $-1$ appears to give a better condition number.} on $\ell^2$, using Lemma~\ref{l:opgrowth}.  For a second-order problem with Dirichlet boundary conditions and $\lambda = 0$ we have $\ell = 0$:
\begin{align*}
    -2 - s \leq -1 \quad \Rightarrow \quad s = 0.
\end{align*}
 For a second-order problem with Neumann boundary conditions and $\lambda = 0$ we have $\ell = 1$:
\begin{align*}
    2 -2 - s \leq -1 \quad \Rightarrow \quad s = 1.
\end{align*}
We demonstrate this on \eqref{eq:boundary-layer} with $\epsilon = 0.05$ in Figure~\ref{fig:GMRES}(A) and with Neumann boundary conditions $u'(-1) = 0, u'(1) = 1$ in Figure~\ref{fig:GMRES}(B).  We see that the condition number is bounded as the discretization is refined.  Consequently, the required number of GMRES iterations required to achieve a relative tolerance of $10^{-14}$ saturates.  

\begin{figure}[tbp]
    \centering
    \begin{subfigure}[b]{0.48\textwidth}
    \centering
    \includegraphics[width=\linewidth]{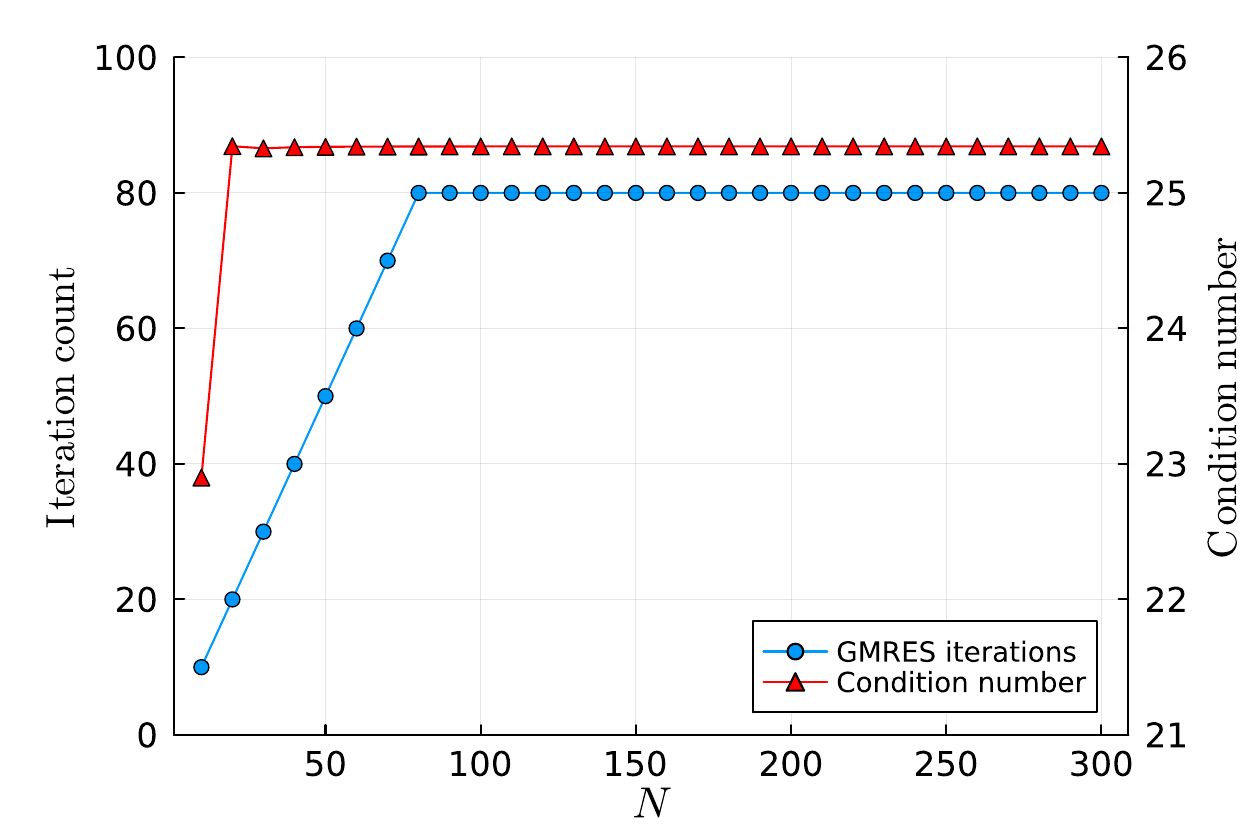}
    \caption{}
    \end{subfigure}
    \begin{subfigure}[b]{0.48\textwidth}
    \centering
    \includegraphics[width=\linewidth]{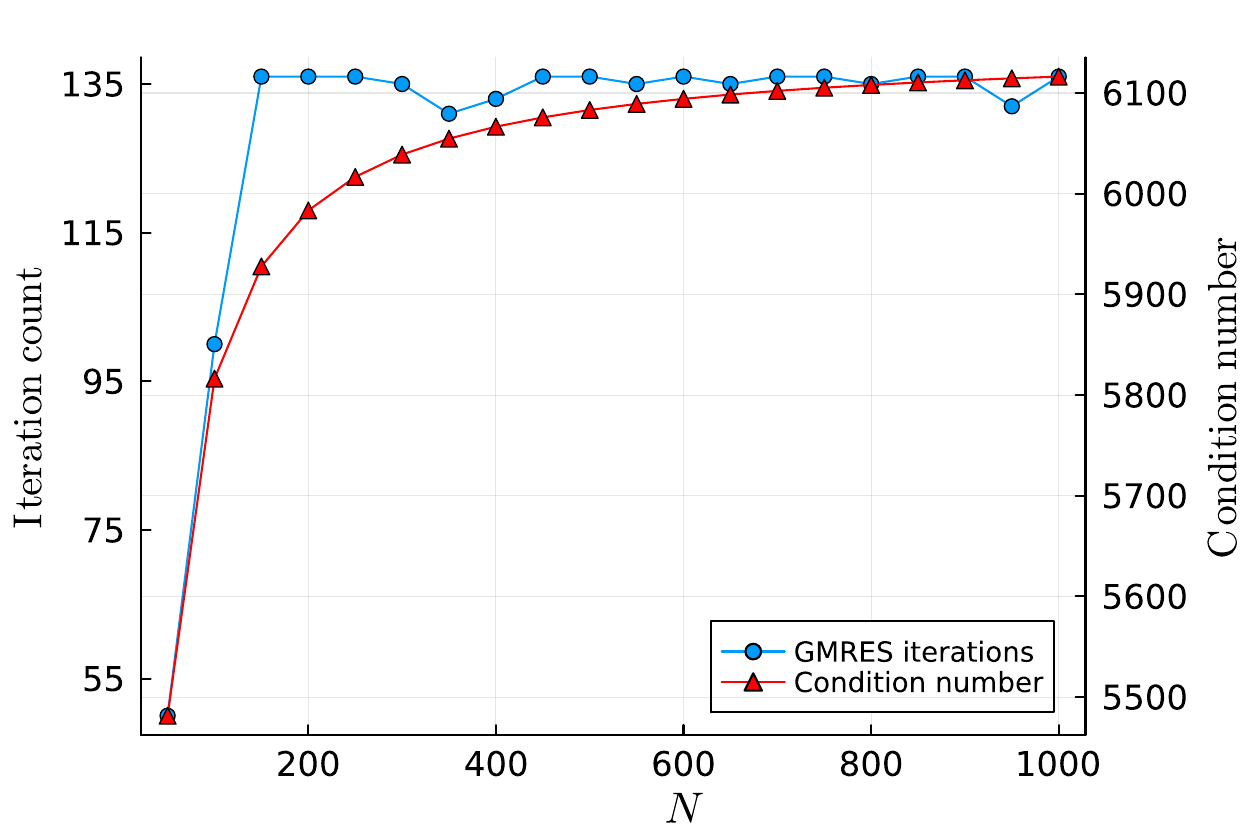}
    \caption{}
    \end{subfigure}
    \begin{subfigure}[b]{0.48\textwidth}
    \centering
    \includegraphics[width=\linewidth]{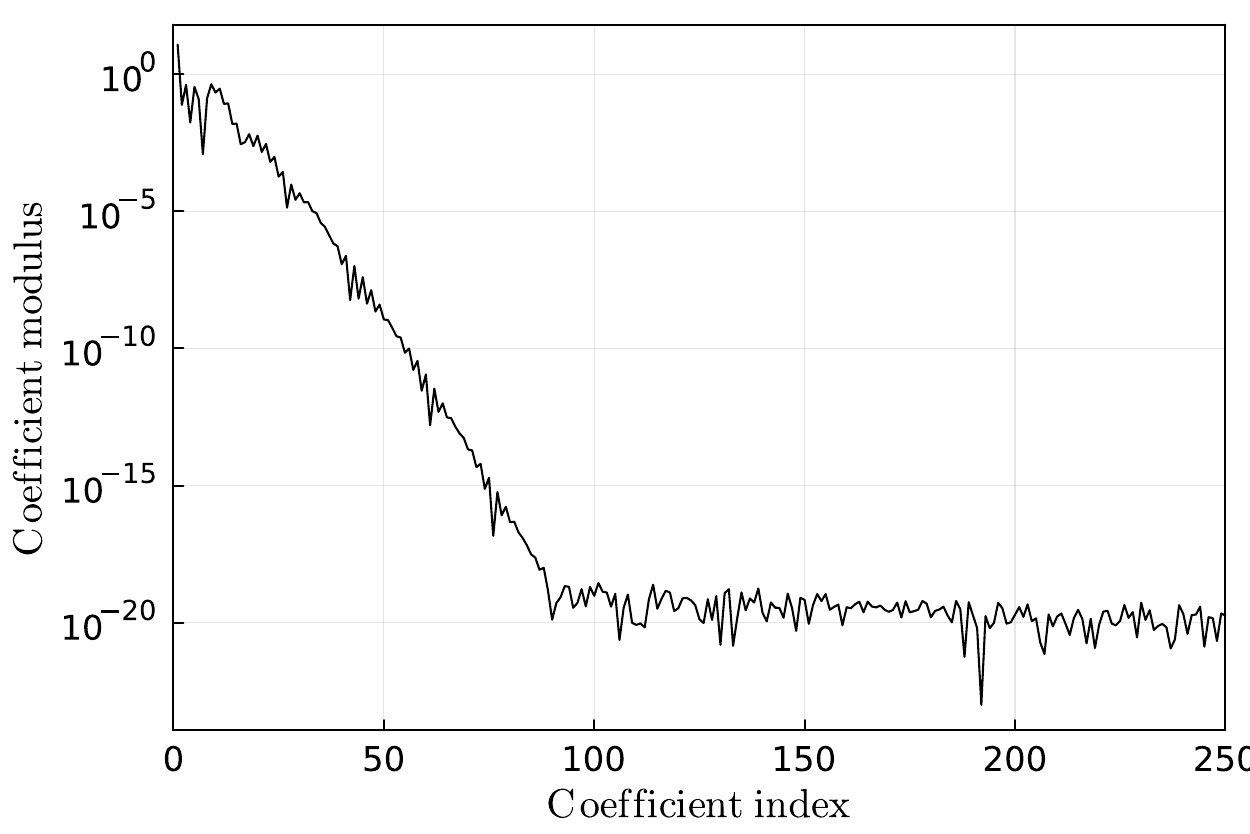}
    \caption{}
    \end{subfigure}
    \caption{(A) The condition number and number of iterations to solve \eqref{eq:boundary-layer} using preconditioned GMRES with $s = 0$.  (B) The condition number and number of iterations to solve \eqref{eq:boundary-layer} with Neumann boundary conditions using preconditioned GMRES with $s = 1$. (C)  The modulus of the first $250$ coefficients for the case of Neumann boundary conditions.  We see that the coefficients saturate below machine precision.}
    \label{fig:GMRES}
\end{figure}

\section{Open questions and future work}\label{sec:open}

The first main open question here is the resolution of Conjecture~\ref{conj}.  This would (1) simplify the proofs given here, (2) possibly give optimal rates of convergence, and (3) potentially provide rigorous justification for the preconditioning given in the last section.  The second main open question is to remove the boundary condition restriction in Theorem~\ref{t:Kras}.

The URC method also raises an important question about the computation of the roots of $p_N(x;\lambda)$ and the application of $\vec F_N(\mu_\lambda)$.  It seems that fast methods \`a la \cite{Hale2012, Bogaert2014a,Townsend2014} could be employed using the asymptotics of the orthogonal polynomials.  Furthermore, the fast application of $\vec F_N(\mu_\lambda)$ would be of use, and one approach would be to extend \cite{Hale2016}.

\section*{Acknowledgements}

The author would like to thank Sheehan Olver and Alex Townsend for helpful input regarding a draft of this paper.  This work was supported in part by DMS-2306438.

\appendix

\section{A review of a theorem of \cite{Krasnoselskii1972}}\label{sec:Kras}

In this section, we adapt \cite[Theorem~15.5]{Krasnoselskii1972} to our setting.    The theorem initially applies in the setting where $a_k = 1$, $\vec b = \vec 0$ and under the assumption that the leading-order operator equation, i.e, $a_j = 0$, $j = 0,1,2,\ldots,k-1$, is uniquely solvable.  We remove the $\vec b = \vec 0$ assumption, but are unable to remove any other restriction.

To state the theorem, let $\mathcal I_N^{P}$ be the polynomial interpolation operator at $N$ distinct nodes $P$ in $\mathbb I$.  

\begin{theorem}\label{t:Kras}
    Suppose $a_k = 1$, $a_j \in C^{0,\alpha}(\mathbb I), j = 0,1,2,\ldots,k-1$, and $f\in C^{0,\alpha}(\mathbb I)$, $\alpha > 0$.  Suppose that with the imposed boundary conditions, the leading-order operator equation and the full operator equation are both uniquely solvable.  If  the Lebesgue constant for $P$ satisfies $\|\mathcal I_{N-k}^{P}\|_\infty = o(\min\{N^{\alpha},N^{1/2}\})$, then for $N$ sufficiently large, the method \eqref{eq:coll} using the nodes $P = P_{N-k}$ as collocation nodes produces a solution $u_N$  that converges to the true solution $u$ of \eqref{eq:bvp}:
    \begin{align*}
        \| u_N - u\|_{H^k(\mathbb I)} \leq C  \left\| u^{(k)} - \mathcal I_{N-k}^P u^{(k)}\right\|_{L^2(\mathbb I)}.
    \end{align*}
\end{theorem}

\begin{proof}
    We first show convergence when $\vec b = \vec 0$.  Let $P_{N-k} = (x_1,\ldots,x_{N-k})$ be the desired nodes.    By the unique solvability of the leading-order problem, we have that the only choice of coefficients such that
    \begin{align*}
        \sum_{j=0}^{k-1} c_j p_j(x;\lambda) 
    \end{align*}
    satisfies the boundary conditions is $c_j = 0 $ for all $j$.  This fact is equivalent to the principal $k\times k$ subblock of $\vec B := \vec S \vec E_{-1}(\lambda,k) + \vec T \vec E_1(\lambda, k)$ being invertible.  Then we select a basis $\vec V_N \in N \times N -k$ for the nullspace of $\vec B$ and consider the discretization, in the notation of \eqref{eq:coll},
    \begin{align*}
        \vec L_P \vec Q_N \vec V_N \vec c_N = \begin{bmatrix} f(x_1) \\ \vdots \\ f(x_{N-k}) \end{bmatrix}.
    \end{align*}
    Furthermore, for $V_N = (v_{ij})$, we find that
    \begin{align*}
        \phi_j(x) = \sum_{i=1}^{N-1} v_{ij} p_{i-1}(x;\lambda),
    \end{align*}
    is a polynomial that satisfies the boundary conditions.

The Green's function operator $\mathcal G$, for the leading-order operator (with the given boundary conditions), induces a bounded linear transformation from $L^2(\mathbb I)$ to $H^k(\mathbb I)$.   Necessarily,
    \begin{align*}
        \mathcal G \phi_j^{(k)} = \phi_j.
    \end{align*}
    So, setting $\vec c_N = (c_j)$ our collocation system can be written as
    \begin{align*}
        \sum_{j=1}^N c_j \left[\phi_j^{(k)}(x_i) + \sum_{\ell=0}^{k-1} a_\ell(x_i) \frac{\D^\ell}{\D x^\ell}\mathcal G \phi_j^{(k)}(x_i) \right] = f(x_i), \quad i = 1,2,\ldots,N-k.
    \end{align*}
    This can be identified with the collocation projection $\mathcal I_{N-k}^{P}, ~P = P_{N-k},$ applied to discretize an operator equation
    \begin{align*}
        (\id + \mathcal K) \psi = f,
    \end{align*}
    where $\mathcal K: L^2(\mathbb I) \to C^{0,\alpha'}(\mathbb I)$, $\alpha' = \min\{\alpha,1/2\}$, and $\id + \mathcal K$ is invertible on $L^2(\mathbb I)$.  That is, we seek the approximate solution
    \begin{align*}
        (\id + \mathcal I_{N-k}^{P}\mathcal K) \psi_N = \mathcal I_{N-k}^{P}f, \quad \psi_N \in \mathrm{span}\{\phi_1^{(k)},\ldots,\phi_{N-k}^{(k)}\}.
    \end{align*}
    We then see that $\mathrm{span}\{\phi_1^{(k)},\ldots,\phi_{N-k}^{(k)}\}$ is simply the span of all polynomials of degree at most $N -k -1$.  Indeed, suppose these functions are linearly dependent.  Then there is a non-trival linear combination that vanishes. This contradicts the assumed unique solvability of leading-order problem.  So, now, we are in the classical framework of projection methods and we claim that
    \begin{align*}
        \|(\id - \mathcal I_{N-k}^{P}) \mathcal K\|_{L^2(\mathbb I)} \to 0,
    \end{align*}
    as $N \to \infty$.  Indeed, it suffices to show that
    \begin{align*}
        \|\id - \mathcal I_{N-k}^{P}\|_{C^{0,\alpha'}(\mathbb I) \to L^2(\mathbb I)} \to 0.
    \end{align*}

    While stronger results are possible, to see this, we note that for $g \in H^1(\mathbb I)$, $g$ can be taken to be $1/2$-H\"older continuous with H\"older constant bounded above by the $H^1(\mathbb I)$ norm of $g$.  Since $\mathcal G$ maps to $H^k(\mathbb I)$, $\frac{\D^\ell}{\D x^\ell}\mathcal G$, $0 \leq \ell \leq k -1$, maps boundedly from $L^2(\mathbb I)$ to $H^1(\mathbb I)$.  Because the coefficient functions are $C^{0,\alpha}(\mathbb I)$ we obtain that
    \begin{align*}
        u \mapsto g := a_\ell \frac{\D^\ell}{\D x^\ell}\mathcal G u
    \end{align*}
    is a bounded operator from $L^2(\mathbb I)$ to $C^{0,\alpha'}(\mathbb I)$, $\alpha' = \min\{\alpha,1/2\}$. Then for $g \in C^{0,\alpha'}(\mathbb I)$, Jackson's theorem gives that the best polynomial approximation $p_{N-k}^*$ of degree $N -k -1$ satisfies
    \begin{align*}
        \|p_{N-k}^* - g\|_\infty \leq C \|g\|_{C^{0,\alpha'}(\mathbb I)} N^{-\alpha'}.
    \end{align*}
   And therefore
    \begin{align*}
        \|\mathcal I_{N-k}^{P}g - g\|_{L^2(\mathbb I)} \leq \sqrt{2} \|\mathcal I_{N-k}^{P}g - g\|_\infty \leq \sqrt{2}C (1 + \|\mathcal I_{N-k}^{P}\|_\infty) \|g\|_{C^{0,\alpha'}(\mathbb I)} N^{-\alpha'}.
    \end{align*}
     With the assumption $ \|\mathcal I_{N-k}^{P}\|_\infty = o(\min\{N^{\alpha},N^{1/2}\})$, Theorem~\ref{t:proj} applies, giving $c, C > 0$ such that for $N$ sufficiently large
    \begin{align}\label{eq:conv}
    c \| \psi - \mathcal I_{N-k}^{P} \psi \|_{L^2(\mathbb I)} \leq \| \psi - \psi_N \|_{L^2(\mathbb I)} \leq C \| \psi - \mathcal I_{N-k}^{P} \psi \|_{L^2(\mathbb I)}.
    \end{align}
    This establishes the required convergence when $\vec b = \vec 0$, after applying $\mathcal G$.

    It remains to treat $\vec b \neq \vec 0$.  To do this, we augment $\vec V_N$ with the first $k$ standard basis vectors
    \begin{align*}
    \tilde {\vec V}_N = \begin{bmatrix} \vec e_1 & \cdots & \vec e_k & \vec V_N \end{bmatrix}.  
    \end{align*}
    We claim that $ \tilde {\vec V}_N$ has linearly independent columns.  If this were not the case, then a non-trivial linear combination of the columns of $\vec V_N$ would give a non-trivial linear combination of the first $k$ standard basis vectors.  But, because the first $k\times k$ principal subblock $\vec B_1$ of $\vec B$ is invertible this contradicts that the columns of $\vec V_N$ are in the nullspace of $\vec B$.  So, the full system one has to consider is
    \begin{align*}
        \begin{bmatrix} 
        \vec B_1 & \vec B_2 \\
        \vec L_{1,P} & \vec L_{2,P} \end{bmatrix} \begin{bmatrix} \vec e_1 & \cdots & \vec e_k & \vec V_N \end{bmatrix} \begin{bmatrix} \vec d_1 \\ \vec d_2 \end{bmatrix}  = \begin{bmatrix} \vec b \\ f(x_1) \\ \vdots \\ f(x_{N-k}) \end{bmatrix}.
    \end{align*}
    This is rewritten as
     \begin{align*}
        \begin{bmatrix} 
        \id_k & \vec B_{1}^{-1} \vec B_2 \\
        \vec L_{1,P} & \vec L_{2,P} \end{bmatrix} \begin{bmatrix} \vec e_1 & \cdots & \vec e_k & \vec V_N \end{bmatrix} \begin{bmatrix} \vec d_1 \\ \vec d_2 \end{bmatrix}  = \begin{bmatrix}\vec B_{1}^{-1} \vec b \\ f(x_1) \\ \vdots \\ f(x_{N-k}) \end{bmatrix}.
    \end{align*}
    By writing out the equations for $\vec d_1$ and $\vec d_2$ we find:
    \begin{align*}
        \vec d_1 &= \vec B_{1}^{-1} \vec b,\\
        \begin{bmatrix} \vec L_{1,P} &  \vec L_{1,P} \end{bmatrix} \vec V_N \vec d_2 &= \begin{bmatrix} f(x_1) \\ \vdots \\ f(x_{N-k}) \end{bmatrix} - \vec L_{1,P} \vec B_{1}^{-1} \vec b.
    \end{align*}
    We recognize $(s_j) = \vec s = \vec B_{1}^{-1} \vec b$ to be the choice of the coefficients $s_0,\ldots,s_{k-1}$ such that
    \begin{align*}
        b(x) := \sum_{j=0}^{k-1} s_j p_j(x;\lambda)
    \end{align*}
    satisfies the boundary conditions.  Then, we recognize the equation for $\vec d_2$ to be the discretization of the boundary-value problem with $\vec b = \vec 0$ and $f$ replaced with $f(x) - \mathcal L b(x)$.    And since solution of this problem is given by $u(x)  - b(x)$ where $u$ is the solution of \eqref{eq:bvp}.  Solving for $\vec d_2$ generates a convergent approximation  \`a la \eqref{eq:conv}.  Since $b$ is a low-degree polynomial, we can add and subtract it within each of these norms, using that $\mathcal I_{N-k}^{P}$ is a projection, to obtain the result.
\end{proof}

\begin{remark}\label{r:zeros}
We then pause to remark that the following choices all give $\|\mathcal I_{N-k}^P\|_\infty = o(N^{-1/2})$:
\begin{itemize}
    \item the Chebyshev first-kind extrema,
    \item the Chebyshev first-kind zeros, and
    \item the roots of $p_{N-k}(x;\lambda)$ for $0 \leq \lambda < 1$.
\end{itemize}
See \cite[p.~336]{Szego1939} for a discussion of the fact that $\|\mathcal I_N^{\mu_\lambda}\|_\infty = O(\max\{\log N, N^{\lambda - 1/2}\})$ and hence the restriction $\lambda < 1$.
\end{remark}

\section{Deferred proofs}\label{sec:defer}

\begin{proof}[Proof of Theorem~\ref{t:main}]
To simplify notation set $f_m = \mathcal I_m^{\rm Ch}f$. Since the top rows vanish identically in the difference, we must compare
\begin{align*}
    \vec T_j(f_m) &:= \begin{bmatrix} \id_{N-k} & \vec 0 & \cdots \end{bmatrix}
    \vec M(f_m;k + \lambda)\vec C_{j + \lambda \to k + \lambda} \vec D_j(\lambda) \vec Q_N,\\
    \tilde{\vec T}_j(f_m) &:=  \left[\vec F_{N-k}(\mu_{\lambda + k}) f_m(P) \vec F_{N-k}(\mu_{\lambda + k})^{-1}\right] \vec F_{N-k}(\mu_{\lambda + k})\vec P_{\lambda + k \to P} \vec C_{\lambda +j \to \lambda + k}\vec D_j(\lambda)\vec Q_N,
\end{align*}
by estimating
\begin{align*}
     \bDelta^{(s)}_{N-k} (\vec T_j(f_m) - \tilde{\vec T}_j(f_m)) \vec Z_N^{-1} \bDelta^{(-s)}_N.
\end{align*}
Then, we move to
\begin{align*}
    \vec F_{N-k}(\mu_{\lambda + k}) f_m(P) \vec F_{N-k}(\mu_{\lambda + k})^{-1} &=   \vec F_{N-k}(\mu_{\lambda + k})\vec P_{\lambda + k \to P} \vec M(f_m;k + \lambda)\vec Q_{N-k}.
\end{align*}
To better express contributions, we block
\begin{align*}
    \vec M(f_m,k + \lambda) = \begin{bmatrix} \vec M_{11} & \vec M_{12} & \vec M_{13} \\
    \vec M_{21} & \vec M_{22} & \vec M_{12} \\
    \vec 0 & \vec M_{32} & \vec M_{33}
    \end{bmatrix}.
\end{align*}
where $\vec M_{11}$ is $N-k \times N-k$, $\vec M_{12}$ is $N-k \times k$, $\vec M_{21}$ is $m \times N -k$. And we block
\begin{align*}
\vec C_{\lambda +j \to \lambda + k}\vec D_j(\lambda)\vec Z^{-1} = \begin{bmatrix} \vec S_{11} & \vec S_{12}\\
\vec S_{21} & \vec S_{22} \\
\vec 0 & \vec S_{32} \end{bmatrix},
\end{align*}
where $\vec S_{11}$ is $N-k \times N$, $\vec S_{21}$ is $k \times N$.  And we note that $\vec S_{21}$ is only non-zero in its upper-right $k -j \times k-j$ subblock.  We then have
\begin{align*}
    \tilde{\vec T}_j(f_m) \vec Z^{-1}_N&= \begin{bmatrix} \id_{N-k} & \vec a_{N-k+1} & \cdots & \vec a_{N-k+m} \end{bmatrix} \begin{bmatrix} \vec M_{11} \\ \vec M_{21} \end{bmatrix} \begin{bmatrix} \id_{N-k} & \vec a_{N-k+1} & \cdots & \vec a_{N} \end{bmatrix} \begin{bmatrix} \vec S_{11} \\ \vec S_{21} \end{bmatrix},
\end{align*}
and
\begin{align*}
    {\vec T}_j(f_m)\vec Z^{-1}_N &= \begin{bmatrix} \vec M_{11} & \vec M_{12} \end{bmatrix} \begin{bmatrix} \vec S_{11} \\ \vec S_{21} \end{bmatrix}.
\end{align*}
Next, we introduce the weight matrices $\bDelta^{(s)}, \bDelta^{(-s)}$ giving
\begin{align*}
\bDelta^{(s)}_{N-k}\tilde{\vec T}_j(f_m)& \vec Z^{-1}_N\bDelta^{(s)}_{N} = \bDelta^{(s)}_{N-k}\begin{bmatrix} \id_{N-k} & \vec a_{N-k+1} & \cdots & \vec a_{N-k+m} \end{bmatrix}\bDelta^{(-s)}_{N-k+m} \\
& \times\bDelta^{(s)}_{N-k+m}\begin{bmatrix} \vec M_{11} \\ \vec M_{21} \end{bmatrix}\bDelta^{(-s)}_{N-k} \\
& \times \bDelta^{(s)}_{N-k} \begin{bmatrix} \id_{N-k} & \vec a_{N-k+1} & \cdots & \vec a_{N} \end{bmatrix} \bDelta^{(-s)}_{N} \\
& \times \bDelta^{(s)}_{N}\begin{bmatrix} \vec S_{11} \\ \vec S_{21} \end{bmatrix} \bDelta^{(-s)}_{N}.
\end{align*}
and
\begin{align*}
    \bDelta^{(s)}_{N-k}{\vec T}_j(f_m)& \vec Z^{-1}_N\bDelta^{(s)}_{N}= \bDelta^{(s)}_{N-k} \begin{bmatrix} \vec M_{11} & \vec M_{12} \end{bmatrix} \bDelta^{(-s)}_{N}  \bDelta^{(s)}_{N}\begin{bmatrix} \vec S_{11} \\ \vec S_{21} \end{bmatrix} \bDelta^{(s)}_{N}.
\end{align*}
We then use $\check{\phantom T}$ to denote all the respective terms after the $\bDelta^{(\pm s)}$ factors have been absorbed: 
\begin{align*}
   \bDelta^{(s)}_{N-k}\tilde{\vec T}_j(f_m) \vec Z^{-1}_N\bDelta^{(-s)}_{N}&= \begin{bmatrix} \id_{N-k} & \check{\vec a}_{N-k+1} & \cdots & \check{\vec a}_{N-k+m} \end{bmatrix} \begin{bmatrix} \check{\vec M}_{11} \\ \check{\vec M}_{21} \end{bmatrix} \begin{bmatrix} \id_{N-k} & \check{\vec a}_{N-k+1} & \cdots & \check{\vec a}_{N} \end{bmatrix} \begin{bmatrix} \check{\vec S}_{11} \\ \check{\vec S}_{21} \end{bmatrix},\\
    \bDelta^{(s)}_{N-k}{\vec T}_j(f_m)\vec Z^{-1}_N\bDelta^{(-s)}_{N} &= \begin{bmatrix} \check{\vec M}_{11} & \check{\vec M}_{12} \end{bmatrix} \begin{bmatrix} \check{\vec S}_{11} \\ \check{\vec S}_{21} \end{bmatrix}.
\end{align*}
From Corollary~\ref{c:bounded}, there exists $C > 0$ such that $\|\check {\vec M}_{ij}\|_{\ell^2} \leq C$.  The entries of $\vec S_{21}$ get inflated by at most $N^s/(N - k)^s = O(1)$ and therefore $\|\check {\vec S}_{21}\|_{\rm F} = O(N^{j-k})$.  Set
\begin{align*}
    \vec A_\ell = \begin{bmatrix}  \check{\vec a}_{N-k+1} & \cdots & \check{\vec a}_{N-k+\ell} \end{bmatrix} 
\end{align*}
and therefore
\begin{align*}
\bDelta^{(s)}_{N-k}(\tilde{\vec T}_j(f_m) - \bDelta^{(s)}_{N-k}{\vec T}_j(f_m))\vec Z^{-1}_N\bDelta^{(s)}_{N} & = (\check{\vec M}_{11} + \vec A_m \check{\vec M}_{21})( \check{\vec S}_{11} + \vec A_k \check{\vec S}_{21}) - \check{\vec M}_{11}\check{\vec S}_{11} - \check{\vec M}_{12} \check{\vec S}_{21}\\
& = \check{\vec M}_{11} \vec A_k \check{\vec S}_{21} +  \vec A_m \check{\vec M}_{21}\vec A_k \check{\vec S}_{21} - \check{\vec M}_{12} \check{\vec S}_{21}  + \vec A_m \check{\vec M}_{21} \check{\vec S}_{11}.
\end{align*}
In $2$-norm, the first three terms are each $O(m N^{j-k})$.  The last term requires further study.  Block
\begin{align*}
    \check{\vec M}_{21}  = \begin{bmatrix} \vec 0 & \vec R \end{bmatrix},
\end{align*}
where $\vec R$ is $m \times m$ and has bounded $2$-norm.  Then, blocking
\begin{align*}
    \check{\vec S}_{11} = \begin{bmatrix} \hat{\vec S}_1 \\ \hat{\vec S}_2 \end{bmatrix},
\end{align*}
where $\hat{\vec S}_2$ is $m \times N$, we have that $\|\hat{\vec S}_2\|_{\rm F} = O((N - m)^{j-k})$ and that gives
\begin{align*}
    \|\vec A_m \check{\vec M}_{21} \check{\vec S}_{11}\|_{\ell^2} = O(m (N - m)^{j-k}).
\end{align*}
In introducing a factor of $\bDelta_N^{(-t)}$ on the right, we see this will add extra decay of $O((N-m)^{-t})$ to $\hat {\vec S}_2$, $\check {\vec S}_{21}$.  The theorem follows
\end{proof}

 \begin{proof}[Proof of Theorem~\ref{t:inf}]
    For (1), one just needs that $\vec M(a_j; k + \lambda)$ is bounded on $\ell^2_s(\mathbb N)$ and $k + \alpha > 2 + s$ is sufficient by Corollary~\ref{c:bounded}.\\

    For (2), by the Fredholm alternative, it suffices to show that the kernel of $\id + \vec K$ is trivial.  So, if \eqref{eq:bvp} is uniquely solvable, but $(\id + \vec K)\vec v = \vec 0$, $\vec v = (v_j)_j$.  Then set
    \begin{align*}
    v(x) = \sum_{j=0}^\infty v_j p_j(x;\lambda).  
    \end{align*}
    It is sufficient to suppose that $\vec v$ is $\ell^2_s(\mathbb N)$ for $s$ sufficiently large so that $v^{(k)}(x)$ is continuous.  So, set $\vec d = \vec D_k(\lambda) \vec v$ and 
    \begin{align*}
    v^{(k)}(x) = \sum_{j=0}^\infty d_j p_j(x;\lambda+k).  
    \end{align*}
    So, if $\vec v \in \ell_{s + k}^2(\mathbb N)$ then $\vec d \in \ell_{s}^2(\mathbb N)$.  And from Lemma~\ref{l:opgrowth}, $p_j(x;\lambda +k ) = O(j^{\lambda + k})$, we require $-s + \lambda + k < -1/2$ and then
    \begin{align*}
        \left|\sum_{j=1}^\infty d_j p_j(x;\lambda+k) \right| \leq \|\vec d\|_{\ell^2_s} \left( \ell(\lambda + k) \sum_{j=1} j^{- 2s + 2\lambda + 2k} \right)^{1/2} < \infty.
    \end{align*}
    Then we conclude, by the unique solvability of \eqref{eq:bvp}, that $\vec v = \vec 0$.

    For (3), by the compactness of $\vec K$ it follows that $\vec Q_N\vec Q_N^T \vec K $ converges in operator norm to $\vec K$ \cite{atkinson}.  Therefore $\id + \vec Q_N\vec Q_N^T \vec K$ is invertible for sufficiently large $N$, $N > N_0$, satisfying
    \begin{align*}
        \|(\id + \vec Q_N\vec Q_N^T \vec K)^{-1}\|_{\ell^2_s}  \leq 2 \|(\id + \vec K)^{-1}\|_{\ell^2_s}.
    \end{align*}  
    Furthermore, the range of $\vec Q_N$ is an invariant subspace for this operator, implying that it must be invertible on this subspace.  And on this subspace it is equal to $\vec L_N^{\rm FS}\vec Z_N^{-1}$ so this operator must also be invertible. Thus
    \begin{align*}
    \|\vec Z_N {\vec L_N^{\rm FS}}^{-1}\|_{\ell_s^2} = \sup_{\substack{\vec u \in \mathrm{ran}\, \vec Q_N \\ \|\vec u\|_{\ell^2_s} = 1}}\|(\id + \vec Q_N\vec Q_N^T \vec K)^{-1} \vec u\|_{\ell^2_s} \leq \|(\id + \vec Q_N\vec Q_N^T \vec K)^{-1}\|_{\ell^2_s},
    \end{align*}
    and (3) follows.

    Then (4) is a consequence of standard theory for projection methods \cite{atkinson}.
\end{proof}

\begin{proof}[Proof of Proposition~\ref{p:coll_stable}]
Consider, as above
\begin{align*}
\tilde{\vec T}_j(a_j) &:=  \left[\vec F_{N-k}(\mu_{\lambda + k}) a_j(P)\right]\vec P_{\lambda + k \to P} \vec C_{\lambda +j \to \lambda + k}\vec D_j(\lambda)\vec Q_N,
\end{align*}
And we examine
\begin{align*}
    \left[\vec F_{N-k}(\mu_{\lambda + k}) (a_j(P) - \tilde a_j(P))\right]\vec P_{\lambda + k \to P}.
\end{align*}
Recall that
\begin{align*}
        \vec F_{N-k}(\mu_{\lambda + k}) = \vec U_{N-k}(\mu_{\lambda + k}) \vec W_{N-k}(\mu_{\lambda + k}).
\end{align*}
So, we can estimate, using Lemma~\ref{l:weights},
\begin{align*}
    \|\bDelta_{N-k}^{(s)}\vec F_{N-k}(\mu_{\lambda + k})\|_{\ell^2} \leq  \|\bDelta_{N-k}^{(s)}\|_{\ell^2} \|\vec U_{N-k}(\mu_{\lambda + k})\|_{\ell^2} \|\vec W_{N-k}(\mu_{\lambda + k})\|_{\ell^2}\leq C N^{s -1/2}. 
\end{align*}
Then, as a crude bound, by Lemma~\ref{l:opgrowth}, using the Frobenius norm as an upper bound
\begin{align*}
    \|\vec P_{\lambda + k \to P} \bDelta^{(-s)}_N\|_{\ell^2}^2 \leq \ell(\lambda + k)^2\sum_{i = 1}^{N-k} \sum_{j = 1}^N j^{2(k + \lambda - s)} \leq C_{\lambda+k,s} N.
\end{align*}
The proposition follows.
\end{proof}

\section{Key ideas from operator theory and projection methods}

In this section, we use script upper-case Roman letters for bounded linear operators between Banach spaces.  We include these results for completeness but point the reader to a proper text \cite{atkinson}.  

\begin{theorem}\label{t:continuity}
Suppose that $\mathcal L \in L(V,W)$ is invertible.  If $\mathcal M \in L(V,W)$ is such that $\|\mathcal L -\mathcal M \|_{V \to W} < \|\mathcal L^{-1}\|_{W \to V}^{-1}$, then $\mathcal M$ is also invertible.  Furthermore, we have the following estimates
    \begin{align*}
         \| \mathcal M^{-1} \|_{W \to V} &\leq \frac{\|\mathcal L^{-1}\|_{W \to V}}{1 - \rho},\\
         \|\mathcal M^{-1} - \mathcal L^{-1} \|_{W \to V} &\leq \frac{\rho}{1 - \rho} \|\mathcal L^{-1}\|_{W \to V},
    \end{align*}
    where $\rho = \|\mathcal L^{-1}\|_{W \to V} \|\mathcal L -\mathcal M \|_{V \to W}$
\end{theorem}

Thus, in the context of the previous theorem, if $\mathcal L u = f$ and $\mathcal M v = f$, we have
\begin{align}\label{eq:diffO}
    \|u - v\|_{V} = O(\|\mathcal L -\mathcal M \|_{V \to W} \|f\|_W).
\end{align}
But one can do much better if one is considering operator equations
\begin{align*}
    (\id + \mathcal K) u = f, \quad (\id + \mathcal P_n \mathcal K) u_n = \mathcal P_n, \quad u_n \in \mathrm{ran}\, \mathcal P_Nf,
\end{align*}
for a projector $\mathcal P_n$.

\begin{theorem}\label{t:proj}
Suppose that $\id + \mathcal K \in L(V)$ is invertible.  Suppose $\| (\id - \mathcal P_n) \mathcal K\|_{V} \to 0$.  Then for $n$ sufficiently large $\id + \mathcal P_n \mathcal K$ is invertible on $\mathrm{ran}\, \mathcal P_N$ and there exists $c,C > 0$ such that
\begin{align*}
    c \|u - \mathcal P_nu \|_V \leq \|u - u_n \|_V \leq  C \|u - \mathcal P_nu \|_V.
\end{align*}
\end{theorem}

\bibliographystyle{abbrv}
\bibliography{library,other}  

\begin{thebibliography}{10}

\bibitem{atkinson}
K.~Atkinson and W.~Han.
\newblock {\em {Theoretical Numerical Analysis}}.
\newblock Springer, New York, NY, 2009.

\bibitem{Aurentz2020}
J.~L. Aurentz and R.~M. Slevinsky.
\newblock {On symmetrizing the ultraspherical spectral method for self-adjoint problems}.
\newblock {\em Journal of Computational Physics}, 410:109383, jun 2020.

\bibitem{Aurentz2017}
J.~L. Aurentz and L.~N. Trefethen.
\newblock {Block Operators and Spectral Discretizations}.
\newblock {\em SIAM Review}, 59(2):423--446, jan 2017.

\bibitem{chebfun}
Z.~Battles and L.~N. Trefethen.
\newblock {An extension of MATLAB to continuous functions and operators}.
\newblock {\em SIAM J. Sci. Comput.}, 25:1743--1770, 2004.

\bibitem{Bogaert2014a}
I.~Bogaert.
\newblock {Iteration-Free Computation of Gauss--Legendre Quadrature Nodes and Weights}.
\newblock {\em SIAM Journal on Scientific Computing}, 36(3):A1008--A1026, jan 2014.

\bibitem{Boyd}
J.~P. Boyd.
\newblock {\em {Chebyshev and Fourier spectral methods}}.
\newblock Dover Publications Inc., Mineola, NY, second edition, 2001.

\bibitem{Canuto1988}
C.~Canuto, M.~Y. Hussaini, A.~Quarteroni, and T.~A. Zang.
\newblock {\em {Spectral Methods in Fluid Dynamics}}.
\newblock Springer Berlin Heidelberg, Berlin, Heidelberg, 1988.

\bibitem{Clenshaw1955}
C.~W. Clenshaw.
\newblock {A note on the summation of Chebyshev series}.
\newblock {\em Mathematics of Computation}, 9(51):118--118, sep 1955.

\bibitem{DeiftOrthogonalPolynomials}
P.~Deift.
\newblock {\em {Orthogonal Polynomials and Random Matrices: a Riemann-Hilbert Approach}}.
\newblock Amer. Math. Soc., Providence, RI, 2000.

\bibitem{Driscoll2015}
T.~A. Driscoll and N.~Hale.
\newblock {Rectangular spectral collocation}.
\newblock {\em IMA Journal of Numerical Analysis}, page dru062, feb 2015.

\bibitem{Du2015}
K.~Du.
\newblock {Preconditioning rectangular spectral collocation}.
\newblock {\em arXiv preprint 11510.00195}, oct 2015.

\bibitem{Fornberg1996}
B.~Fornberg.
\newblock {\em {A Practical Guide to Pseudospectral Methods}}.
\newblock Cambridge University Press, jan 1996.

\bibitem{Forster1993}
K.-J. F{\"{o}}rster.
\newblock {Inequalities for ultraspherical polynomials and application to quadrature}.
\newblock {\em Journal of Computational and Applied Mathematics}, 49(1-3):59--70, dec 1993.

\bibitem{GautschiOP}
W.~Gautschi.
\newblock {\em {Orthogonal Polynomials: Applications and Computation}}.
\newblock Oxford University Press, 2004.

\bibitem{Hale2012}
N.~Hale and A.~Townsend.
\newblock {Fast and Accurate Computation of Gauss-Legendre and Gauss-Jacobi Quadrature Nodes and Weights}.
\newblock {\em SIAM J. Sci. Comput.}, 35(2):A652--A674, 2013.

\bibitem{Hale2016}
N.~Hale and A.~Townsend.
\newblock {A fast FFT-based discrete Legendre transform}.
\newblock {\em IMA Journal of Numerical Analysis}, 36(4):1670--1684, oct 2016.

\bibitem{ErvandKogbetliantz1924}
E.~Kogbetliantz.
\newblock Recherches sur la sommabilité; des séries ultra-sphériques par la méthode des moyennes arithmétiques.
\newblock {\em Journal de Mathématiques Pures et Appliquées}, 3:107--188, 1924.

\bibitem{Krasnoselskii1972}
M.~A. Krasnosel'skii, G.~M. Vainikko, P.~P. Zabreiko, Y.~B. Rutitskii, and V.~Y. Stetsenko.
\newblock {\em {Approximate Solution of Operator Equations}}.
\newblock Springer Netherlands, Dordrecht, 1972.

\bibitem{KuijlaarsInterval}
A.~B.~J. Kuijlaars, K.~T.-R. McLaughlin, W.~{Van Assche}, and M.~Vanlessen.
\newblock {The Riemann--Hilbert approach to strong asymptotics for orthogonal polynomials on $[-1,1]$}.
\newblock {\em Advances in Mathematics}, 188(2):337--398, 2004.

\bibitem{Lorentz1971}
G.~G. Lorentz and K.~L. Zeller.
\newblock {Birkhoff Interpolation}.
\newblock {\em SIAM Journal on Numerical Analysis}, 8(1):43--48, 1971.

\bibitem{DLMF}
F.~W.~J. Olver, D.~W. Lozier, R.~F. Boisvert, and C.~W. Clark.
\newblock {\em {NIST Handbook of Mathematical Functions}}.
\newblock Cambridge University Press, 2010.

\bibitem{Olver2020}
S.~Olver, R.~M. Slevinsky, and A.~Townsend.
\newblock {Fast algorithms using orthogonal polynomials}.
\newblock {\em Acta Numerica}, 29:573--699, may 2020.

\bibitem{Olver2013}
S.~Olver and A.~Townsend.
\newblock {A Fast and Well-Conditioned Spectral Method}.
\newblock {\em SIAM Review}, 55(3):462--489, jan 2013.

\bibitem{Petras1996}
K.~Petras.
\newblock {An asymptotic expansion for the weights of Gaussian quadrature formulae}.
\newblock {\em Acta Mathematica Hungarica}, 70(1-2):89--100, 1996.

\bibitem{Ross2023}
I.~M. Ross, R.~J. Proulx, and C.~F. Borges.
\newblock {A universal Birkhoff pseudospectral method for solving boundary value problems}.
\newblock {\em Applied Mathematics and Computation}, 454:128101, oct 2023.

\bibitem{GMRES-original}
Y.~Saad and M.~H. Schultz.
\newblock {GMRES: a generalized minimal residual algorithm for solving nonsymmetric linear systems}.
\newblock {\em SIAM J. Sci. Statist. Comput.}, 7(3):856--869, 1986.

\bibitem{Shen2011}
J.~Shen, T.~Tang, and L.-L. Wang.
\newblock {\em {Spectral Methods}}, volume~41 of {\em Springer Series in Computational Mathematics}.
\newblock Springer Berlin Heidelberg, Berlin, Heidelberg, 2011.

\bibitem{Szego1939}
G.~Szegő.
\newblock {\em {Orthogonal Polynomials}}, volume~23 of {\em Colloquium Publications}.
\newblock American Mathematical Society, Providence, Rhode Island, dec 1939.

\bibitem{Townsend2014}
A.~Townsend, T.~Trogdon, and S.~Olver.
\newblock {Fast computation of Gauss quadrature nodes and weights on the whole real line}.
\newblock {\em IMA Journal of Numerical Analysis}, 36(1):337--358, oct 2014.

\bibitem{TrefethenSpectral}
L.~N. Trefethen.
\newblock {\em {Spectral methods in MATLAB}}.
\newblock Society for Industrial and Applied Mathematics, Philadelphia, PA, 2000.

\bibitem{TrefethenATAP}
L.~N. Trefethen.
\newblock {\em {Approximation Theory and Approximation Practice, Extended Edition}}.
\newblock Society for Industrial and Applied Mathematics, Philadelphia, PA, jan 2019.

\bibitem{Trogdon2023}
T.~Trogdon.
\newblock {On the convergence of spectral methods involving non-compact operators}.
\newblock {\em arXiv preprint 2304.14319}, 2023.

\bibitem{Trogdon2024}
T.~Trogdon.
\newblock {https://github.com/tomtrogdon/URCMethod.jl}, 2024.

\bibitem{Wang2014}
L.-L. Wang, M.~D. Samson, and X.~Zhao.
\newblock {A Well-Conditioned Collocation Method Using a Pseudospectral Integration Matrix}.
\newblock {\em SIAM Journal on Scientific Computing}, 36(3):A907--A929, jan 2014.

\bibitem{Weideman2000}
J.~A. Weideman and S.~C. Reddy.
\newblock {A MATLAB differentiation matrix suite}.
\newblock {\em ACM Transactions on Mathematical Software}, 26(4):465--519, dec 2000.

\bibitem{Xu2016}
K.~Xu and N.~Hale.
\newblock {Explicit construction of rectangular differentiation matrices}.
\newblock {\em IMA Journal of Numerical Analysis}, 36(2):618--632, apr 2016.

\bibitem{Zhang2008}
Z.~Zhang.
\newblock {Superconvergence of a Chebyshev Spectral Collocation Method}.
\newblock {\em Journal of Scientific Computing}, 34(3):237--246, mar 2008.

\end{thebibliography}

\end{document}